\documentclass[11pt,a4paper]{article}
\usepackage[utf8]{inputenc}
\usepackage{amsmath}
\usepackage{amsfonts}
\usepackage{amssymb}
\usepackage{graphicx}
\usepackage{color}
\usepackage[left=2cm,right=2cm,top=2cm,bottom=2cm]{geometry}
\usepackage{amsthm,dsfont}
\usepackage{mathtools}
\usepackage[makeroom]{cancel}
\usepackage{ulem}
\usepackage{url}
\usepackage{stackrel}

\usepackage{pgffor}
\usepackage{multicol,multirow}
\usepackage{algorithm}
\usepackage[noend]{algpseudocode}
\makeatletter
\def\BState{\State\hskip-\ALG@thistlm}

\usepackage[pdftex,
bookmarks,
bookmarksnumbered=true,
bookmarksopen=true,
colorlinks,
citecolor=green,
filecolor=black,
linkcolor=red,
urlcolor=green,
pdfauthor={López-Lopera, Andrés Felipe},
pdftoolbar=true,
pdfmenubar=true,
pdffitwindow = true
]{hyperref}

\usepackage{natbib} 
\newtheorem{theorem}{Theorem}[section]
\newtheorem{definition}{Definition}[section]
\newtheorem{proposition}{Proposition}[section]
\newtheorem{remark}{Remark}[section]
\newtheorem{condition}{Condition}[section]
\newtheorem{lemma}{Lemma}[section]

\newcommand{\Hilb}{\mathcal{H}}
\newcommand{\HF}{\Hilb_F}
\newcommand{\HN}{\Hilb_N}
\newcommand{\HNF}{\Hilb_{N, F}}

\newcommand{\ENhat}{\widehat{E}_N}
\newcommand{\FNhat}{\widehat{F}_N}

\newcommand{\addvar}[1]{\, + \,}

\author{}
\date{}

\definecolor{green}{rgb}{0, 0.5, 0}

\title{Error Bounds for a Kernel-Based Constrained Optimal Smoothing Approximation}

\author{
    Laurence Grammont$^{1,\ast}$, Fran\c{c}ois Bachoc$^{2}$ and Andr\'es F. L\'opez-Lopera$^{3}$ \\\\
    \small 
    ${}^{1}$Univ. Jean Monnet, CNRS, ECL, INSA Lyon, UCB Lyon 1, ICJ UMR5208, F-42023 St-\'Etienne, France. \\
    \small ${}^{2}$Institut de Math\'ematiques de Toulouse, Univ. Paul Sabatier, F-31062 Toulouse, France. \\ 
    \small ${}^{3}$Univ. Polytechnique Hauts-de-France, CERAMATHS, F-59313 Valenciennes, France.\\
    \small ${}^{\ast}$Corresponding author
}

\begin{document}

\maketitle

\begin{abstract}
	This paper establishes error bounds for the convergence of a piecewise linear approximation of the constrained optimal smoothing problem posed in a reproducing kernel Hilbert space (RKHS). This problem can be reformulated as a Bayesian estimation problem involving a Gaussian process related to the kernel of the RKHS. Consequently, error bounds can be interpreted as a quantification of the maximum a posteriori (MAP) accuracy. To our knowledge, no error bounds have been proposed for this type of problem so far. The convergence results are provided as a function of the grid size, the regularity of the kernel, and the distance from the kernel interpolant of the approximation to the set of constraints. Inspired by the MaxMod algorithm from recent literature, which sequentially allocates knots for the piecewise linear approximation, we conduct our analysis for non-equispaced knots. These knots are even allowed to be non-dense, which impacts the definition of the optimal smoothing solution and our error bound quantifiers. Finally, we illustrate our theorems through several numerical experiments involving constraints such as boundedness and monotonicity.
\end{abstract}


\section{Introduction}
\label{sec:intro}
Consider a nonempty set $\Omega$ of $\mathbb{R}^d$ and a  set $E$ of   functions from $\Omega$ to $\mathbb{R}$. 
Given data $(x_i,y_i)_{i=1}^n \in \Omega \times \mathbb{R} $,  the   {\bf smoothing problem} is  to find a function $\widehat{u}\in {\cal H}$ solution of
\begin{equation}\label{smoothing}
	\min_{u\in {\cal H}  }
	\|u\|_{\cal H}^2+\displaystyle \frac{1}{\tau} \sum_{i=1}^{n}(u(x_i)-y_i)^2,
\end{equation} 
where ${\cal H}$ is the reproducing kernel Hilbert space (RKHS) defined by a kernel $K$ on $\Omega \times \Omega$, with $\tau >0$. Hence, ${\cal H}$ is a Hilbert space included in $E$, and we let $\| \cdot \|_{\cal H}$ be its Hilbert norm. 
Considering an RKHS allows the solution of \eqref{smoothing} to be interpreted as a Bayesian estimator, involving a Gaussian process (GP) $Y$ with covariance function $K$:  $K(x,x')=\operatorname{cov}(Y(x),Y(x'))$.
We refer for instance to \citep{stein1999interpolation,Rasmussen2005GP,karvonen2023asymptotic} for references on GPs.
In the Bayesian framework, $\tau > 0$ is the noise variance term.
In \citep{wahba1}, it has been proven that $\widehat{u}$ is the mean of the GP $Y$ conditionally to noisy observations:
\begin{equation*}\label{estimation}
	\widehat{u}(t)  = \mathds{E} [Y(t) | Y_1=y_1,\ldots,Y_n=y_n], 
\end{equation*}
where $\mathds{E}$ denotes the expectation of random variables. For the noisy case, the Bayesian model is $Y_i = Y(x_i) + {\cal E}_i$ for all $i = 1, \ldots, n$, where ${\cal E} = ({\cal E}_i)_i \sim {\cal N}(0,\tau I)$ is an independent centered Gaussian vector. %with noise variance $\tau > 0$. 
Here, ${I}$ is the $n \times n$ identity matrix. %For simplicity, we omit the subscripts to the identity matrices as their sizes will be clear from the context. 
The solution $\widehat{u}$ is then given by
$$
\widehat{u}(t) = {k}_n^{\top}(t) \left({K}_n+\tau {I}\right)^{-1} {y},
$$
where ${k}_n(t) = [K(t, x_1), \ldots, K(t,x_n) ]^\top$, ${K}_n = \left(K\left(x_i,x_j\right)\right)_{1\leq i,j\leq n}$ and ${y} = [y_1, \ldots, y_n]^\top$.

If a constraint is added to \eqref{smoothing},  given by a closed convex set $C$ of functions, we obtain  the {\bf constrained smoothing problem} of   finding  a function $\widehat{u}$, in ${\cal H}\cap C$, solution of 
\begin{equation}\label{constrained} 
	\quad \min_{u\in {\cal H}\cap C }
	~
	\|u\|_{\cal H}^2+\displaystyle \frac{1}{\tau} \sum_{i=1}^{n}(u(x_i)-y_i)^2.
\end{equation} 
This problem can be rewritten as  a constrained GP model so that the solution can be interpreted as a Bayesian estimation \citep{GMB}. Two important examples in practice are when $C$ is composed of bounded or componentwise monotonic functions. We refer to~\citep{bellec2018sharp,durot2002sharp,Cousin2016KrigingFinancial,durot2018limit,groeneboom2014nonparametric,Golchi2015MonotoneEmulation,groeneboom2001estimation,hornung1978monotone,lin2014bayesian,LopezLopera2019lineqGPNoise,LopezLopera2017FiniteGPlinear,LopezLopera2019GPCox,maatouk2017gaussian,Riihimaki2010GPwithMonotonicity,Zhou2019ProtonConstrGPs} for consideration of these constraints with GPs and more generally in statistics.

Unlike the unconstrained smoothing problem in~\eqref{smoothing}, there is no explicit expression for the solution in the constrained case, thus a numerical approximation of $\widehat{u}$ is required. For clarity, we restrict our study to the one-dimensional setting $\Omega = [0,1]$. Nevertheless, as explained in~Remark~\ref{rem:multiD}, the techniques we develop can be extended to the general $d$-dimensional case, albeit with more cumbersome notations.

A fruitful approach for numerical approximation is to consider piecewise linear finite-dimensional kernels, RKHSs and GPs~\citep{Bachoc2010cMLE,Cousin2016KrigingFinancial,GMB,LopezLopera2017FiniteGPlinear,maatouk2017gaussian,Zhou2019ProtonConstrGPs}. In particular, we consider the approximate solution $\widehat{u}_N$ in the RKHS defined by $K_N$, the covariance function of a finite-dimensional GP $Y_N$ approximating the GP $Y$. Here, $N \in \mathbb{N}$ is the number of knots defining the piecewise linear approximation. The function $\widehat{u}_N$ is then the solution to a \textbf{constrained discretized smoothing problem}, and also the maximum a posteriori (MAP) of the posterior distribution of the constrained finite-dimensional GP $Y_N$. In~\citep{GMB}, it is shown that $\widehat{u}_N$ converges to $\widehat{u}$ as $N \to \infty$ for fixed data $(x_i,y_i)_{i=1}^n$, however, no error bounds are provided for this convergence. More generally, to our knowledge, no general error bounds have been provided for numerical approximations of the constrained optimal smoothing problem or for the equivalent formulation with the MAP. This is the ambitious aim of this paper.

An error estimation is always highly dependent on regularity, which is related to the function space in which the exact solution is sought. 
In this paper, this function space is determined by the kernel $K$ that we assume to be $\beta$-H\"older, $0 < \beta \leq 1$.
We note that H\"older-continuity is a very standard regularity measure in statistics and machine learning for functions that are not necessarily differentiable~(see, e.g., \citep{pmlr-v75-locatelli18a}). In addition, our error bound construction would not benefit from a stronger regularity that Lipschitzness ($1$-H\"older continuity), because of piecewise linearity (see Remark~\ref{rem:regularity}).

To provide as much generality as possible, we allow for non-equispaced knots defining the finite-dimensional approximation, and we even allow the sequence of knots not to be dense in the input space. Non-equispaced knots enable higher accuracy for a given computational budget $N$~\citep{bachoc2022sequential,LopezLopera2022ACGP} and can be selected automatically by the MaxMod algorithm introduced in~\citep{bachoc2022sequential}. Furthermore, the convergence proof of MaxMod includes an intermediary step analyzing convergences for non-dense knots. This justifies our consideration of non-dense knots when providing error bounds. To account for non-dense and non-equispaced knots, we measure the asymptotic density of the $N$ knots by the specific grid size $\delta_N$ defined in~\eqref{delta_N}. Furthermore, for non-dense knots, the limit function to $\widehat{u}_N$ is not $\widehat{u}$ as in~\eqref{constrained}, but the function $\widehat{u}_F$ defined in~\eqref{eq:constrainedSpline}, which depends on the closure set $F$ of the sequence of knots. This definition relies on the notion of multiaffine extension introduced in~\citep{bachoc2022sequential}. Naturally, when $F = [0,1]$, we have $\widehat{u}_F = \widehat{u}$.

Ultimately, the decay rate of our error bound depends on $\beta$ (i.e. the regularity of the kernel), the grid size $\delta_N$, and a third factor we have highlighted: the distance of the kernel interpolant of the approximate solution to the set of constraints $C_F$ (defined in \eqref{eq:CF}). In the following, this third factor is denoted as $\alpha_N$ (see definition in~\eqref{alphaN}). We demonstrate that $\alpha_N$ tends to zero (see Remark~\ref{rem:alphaN}), though we leave open the problem of quantifying its decay rate in the general setting.

Our final error bounds are provided in Theorems~\ref{label:UB:no:extra:assumption} and \ref{th4}. In Theorem~\ref{label:UB:no:extra:assumption}, we consider the case where $\alpha_N = 0$, meaning the kernel interpolant of the approximate solution satisfies the constraints $C_F$. Here, the error bound is of order $\mathcal{O}(\delta_N^{\beta/4})$. In Theorem~\ref{th4}, we address the most general case, where this kernel interpolant does not necessarily meet the constraints. In this scenario, the error bound also depends on $\alpha_N$.

In numerical experiments, we illustrate and validate our theoretical results through various synthetic examples that account for different types of inequality constraints (e.g., boundedness and monotonicity) and regularity assumptions (i.e. smoothness of the kernel). We examine both scenarios where the knots are dense and where they are not. Numerically, we confirm the convergence as $N \to \infty$, observing a faster convergence for larger regularity $\beta$.

The paper is organized as follows. Section~\ref{sec:cosp} presents the constrained optimal smoothing problem, introducing $\widehat{u}_F$, our regularity indicators and the multiaffine extension, and stating various of their properties. Section~\ref{section:discrete:constrained:optimal:smoothing} discusses the numerical approximation of constrained optimal smoothing, denoted as $\widehat{u}_{N}$, and states its existence and unicity. Section~\ref{section:quantitative:properties} focuses on the quantitative properties of the set of approximants involved in $\widehat{u}_{N}$, which are essential for the final bounds.
%    $\HNF$ which is needed in the elaboration of the error estimations. 
Sections~\ref{sec:errorBounds} and~\ref{sec:numexp} provide the error bounds (with Theorems~\ref{label:UB:no:extra:assumption} and \ref{th4}) and the numerical experiments, respectively. 
%  Section 5 is devoted to the construction of the error bound. It involves the quantity measuring the smoothness of the solution (related to the smoothness of the kernel), the special meshsize and a third  factor that we have highlighted: the distance of the kernel interpolant related to the grid nodes of the approximate solution to the set of constraints $C_F$.  
Section~\ref{conclude} concludes the paper.

Several proofs are included in the main text to elucidate the construction of the error bounds. Additional proofs, primarily technical or containing pre-existing concepts for completeness, are in Appendix~\ref{sectionproofs}.

%%%%%%%%%%%%%%%%%%%%%%%%%%%%%%%%%%%%%%%%%%%%%%%%%%%%%%%%%%%%%%%%%%%%%%%%%%%%%%%%%%%%%%%%%%%%%%%%%%%%%%%%%%%%%%%%%%%%%%%
\section{Framework for the constrained optimal smoothing problem}
\label{sec:cosp}
%%%%%%%%%%%%%%%%%%%%%%%%%%%%%%%%%%%%%%%%%%%%%%%%%%%%%%%%%%%%%%%%%%%%%%%%%%%%%%%%%%%%%%%%%%%%%%%%%%%%%%%%%%%%%%%%%%%%%%

In this section, we present the tools needed to pose the constrained optimal smoothing problem properly. Then we prove, with classical arguments, that the problem has a unique solution with a regularity that we outline. 
%\andres{We need to put in place new indicators and new techniques to build up our error estimate.}{(I don't see the link with the previous paragraph)} 

To ensure that our approach is easiest to understand, we have chosen to focus on the one-dimensional case. 
Remark~\ref{rem:multiD}
discusses the extension to the multi-dimensional case. Hence, here, $E$ is the set of real valued continuous functions on $\Omega=[0,1]$.
Let $F$ be a compact subset of $[0,1]$ containing $0$ and $1$. $F$ will be the closure of the set of knots required in the discretization process of the approximation (see Section~\ref{section:discrete:constrained:optimal:smoothing}). As discussed in Section~\ref{sec:intro}, allowing for a general $F$ that does not coincide with the entire set $\Omega$ is an intermediary step in the convergence proof of MaxMod in \citep{bachoc2022sequential}.
This justifies our interest here to allow for a general $F$.

We define  $E_F$ as the set of real-valued continuous functions restricted to $F$, endowed with the supremum norm:
$$
\| u \|_{\infty}=\max_{t\in F} | u(t) |,
$$
for $u\in E_F$.  
%$E_F=E_{|F}$
Let $\HF$ and ${\cal H}$ be the RKHS, both given by the kernel $K$ on $F$ and on $[0,1]$.
The Hilbert scalar product and norm for ${\cal H}$ are written as $\langle \cdot , \cdot \rangle_{\cal H} $ and $\| \cdot \|_{\cal H} $. Similarly, for $\HF$ they are written as $\langle \cdot , \cdot \rangle_{\HF} $ and $\| \cdot \|_{\HF} $.
Let us notice that $\HF= {\cal H}_{|F}$, where 
\[
{\cal H}_{|F}
=
\{
g : F \to \mathbb{R} ; 
\exists  f \in {\cal H}  ~ \mbox{s.t.} ~
\forall x \in F, g(x) = f(x)
\},
\]
see (\citep{berlinet2011reproducing}, Theorem 6). Both ${\cal H}$ and $\HF$ rely on the reproducing kernel $K$, which is always symmetric and positive semi-definite.   

The regularity of $K$ has a major influence on the error bounds. Here, regularity is measured by H\"older-continuity (recall the corresponding discussion in Section~\ref{sec:intro} and Remark~\ref{rem:regularity}). 
Let $\beta$ be a real number such that $0 < \beta \leq 1$.

\begin{definition}
	A function $f$ is $\beta$-H\"older continuous on $[0,1]$ if there exists a constant $c_f > 0$ such that, $\forall s,t \in [0,1]$, 
	$$
	| f(s)-f(t) | \leq c_f | s-t |^{\beta}.
	$$
\end{definition}

\begin{condition}
	\label{cond:contKernel}
	The reproducing kernel $K$ is  $\beta$-H\"older continuous with respect to both its inputs. That is, there exists a constant $c_K > 0$ such that 
	\begin{equation*}\label{Kbetaholder}
		\forall u,s,t \in [0,1], \quad      | K(u,s)-K(u,t)| \leq c_K | s-t |^{\beta}.
	\end{equation*}
\end{condition}

In the following, we will define quantities which are capable of reflecting the regularity of the problem and which will be useful in setting the error bounds. Let us define the modulus of continuity of a function defined and continuous on a compact subset $S$ of $[0,1]$:
\begin{align} \label{eq:Mf:delta}
	M_f(\delta) =
	\left\{
	\begin{array}{cl}
		\displaystyle \sup_{s, t \in S, \ | s-t|\leq \delta}   | f(s)-f(t)|  & \text{if } \delta \leq 1, \\
		M_f(1) &\text{if } \delta \geq 1.
	\end{array}
	\right.
\end{align}
%If $\delta \leq 1$,
%\begin{equation}\label{Mf}
	%    M_f(\delta) = \displaystyle\sup_{s, t \in S, \ | s-t|\leq \delta}   | f(s)-f(t)|
	%\end{equation}
%with the convention. If $\delta \geq 1$,
%\begin{equation} 
	%   M_f(\delta) =  M_f(1). 
	%\end{equation}
Let us define the following quantity intended to be an indicator of regularity:      
\begin{equation}\label{eq:psif}
	\Psi_f(\delta) = \displaystyle\sup_{t\geq 1}    \dfrac{M_f(t\delta)}{t}.
\end{equation} 
The indicator $\Psi_f$ is adapted to the hat functions considered in Section~\ref{section:discrete:constrained:optimal:smoothing} for the numerical approximation, and is then involved in the proofs of our main error bounds. 
% Let us notice that if   $f\in E$ then $\Psi_f(\delta) $ is finite. Indeed we have  $\dfrac{M_f(t\delta)}{t}\leq M_f(1)$, so that 
%  Suppose that $\delta \leq 1$. 
The following proposition will clarify its order of magnitude according to the regularity of $f$. %$S$ denotes any compact subset of $[0,1]$, $c_f$ is defined in \Cref{betaholder} if  $f$ is $\beta$-H\"older continuous and $d_f=\displaystyle \sup_{t \in S} | f'(t) | $ if $f$ is continuously differentiable. 

%%%%%%%%%%%%%%%%%%%%%%%%%%%%%%%%%%%%%%%%%%%%%%%% PROP 1 
\begin{proposition}\label{prop:psif}
	Let $S$ be a compact subset of $[0,1]$.
	If $f$ is continuous on $S$, 
	\begin{displaymath}\label{psiE} 
		\displaystyle \lim_{\delta \to 0} \Psi_f(\delta)=0.
	\end{displaymath}
	If $f$ is $\beta$-H\"older continuous on $S$,  
	\begin{displaymath}\label{psibeta} 
		\Psi_f(\delta) \leq c_f \delta^{\beta}.
	\end{displaymath}
	% If $f$ is continuously differentiable on $S$,  
	%  \begin{equation}\label{psiconti}
		%   \Psi_f(\delta) \leq d_f \delta 
		%\end{equation}
\end{proposition}
%%%%%%%%%%%%%%%%%%%%%%%%%%%%%%%%%%%%%%%%%%%%%%%%%%%%%%%%%%%%%

\begin{proof} 
	The proof is postponed to Appendix~\ref{sectionproofs}.
\end{proof}

Then, we define the multi-affine extension announced in Section~\ref{sec:intro}. This extension will allow us to define the constrained optimal smoothing problem for functions defined on $F$ rather than on $[0,1]$. Then, this extension will allow us to define the limit function $\widehat{u}_F$ discussed in Section~\ref{sec:intro}, to which this paper provides upper bounds.
%To be able to write the problem and to define the constraints, we need to recall the piecewise affine extension from $E_F$ to $E$, corresponding to the multi-affine extension of \citep{bachoc2022sequential} in the one-dimensional case. 
\begin{definition} \label{def:multi-affine:extension}
	For  $t \in [0,1]$,
	\begin{itemize}
		\item If $t \in F$, then define $t^- = t^+ = t$ and $w_-(t) = w_+(t) = 1/2$.
		\item If $t \not \in F$, then define $t^- =  \sup \{ x, x \in F, x \leq t\}$,  $t^+ =  \inf \{ x, x \in F, x \geq t\}$, and 
		\begin{align*}
			w_-(t) = (t^+-t)/(t^+ - t^-), \qquad
			w_+(t) = (t-t^-)/(t^+ - t^-).
		\end{align*}
	\end{itemize}
	Define the operator $P: E_F \to E$ as follows. For all $u \in E_F$, 
	\begin{equation*}
		P(u)(t) = u(t^-)w_-(t) +  u(t^+)w_+(t),
	\end{equation*}
	and call $P(u) \in E$ the multi-affine extension of $u$.
\end{definition}
We remark that, in Definition \ref{def:multi-affine:extension}, the multi-affine extension could also be called the affine extension. Nevertheless,  it is extended to the general multi-dimensional case in \citep{bachoc2022sequential}, where the name multi-affine extension is appropriate. Hence, for clarity, we will refer to $P(u)$ in Definition \ref{def:multi-affine:extension} as a multi-affine extension also in the one-dimensional exposition of this paper. 

\begin{proposition}  
	$P(u)$ is the unique function in $E$  equal to  $u$ on $F$  and affine  on the  intervals of $[0,1] \backslash F$.  Moreover the map $u \to P(u)$ is linear and 1-Lipschitz from $E_F$ to $E$ equipped with the supremum norm. In particular, it preserves uniform convergence.
\end{proposition}
\begin{proof}
	See (\citep{bachoc2022sequential}, Proposition 4.4).   
\end{proof}

We recall the constraint set $C$, which is assumed to be closed and convex in $E$. 
Let us then define the set of functions on $F$ which multi-affine extensions are in $C$.  
\begin{equation} \label{eq:CF}
	C_F = \{f \in E_F:   P(f) \in C\}.
\end{equation}
Since $P$ is linear and $C$ is convex,  $C_F \ \mbox{is a convex set of} \ E_F$.
\begin{condition} \label{cond:HfintCf}
	$ \HF \cap C_F \neq \emptyset. $
\end{condition}
Condition~\ref{cond:HfintCf} guarantees the compatibility of the constraints with the kernel $K$ and set $F$. This means that there is a function in $\HF$ which multi-affine extension satisfies the constraints. It is clear that our definition of $\widehat{u}_F$ in \eqref{eq:constrainedSpline} below needs this condition, where  $\widehat{u}_F$ is discussed in Section~\ref{sec:intro}. For the standard cases where $C$ is the a set  of bounded, monotonous or convex functions, Condition~\ref{cond:HfintCf} holds for many standard classes of kernels. In particular, when $C$ is the set of functions bounded in $[\ell,u]$ for fixed $- \infty < \ell < u < \infty$, then Condition~\ref{cond:HfintCf} holds for all Mat\'ern kernels \citep{stein1999interpolation}. Indeed, the constant function $(\ell+u)/2$ satisfies the constraints and belongs to the corresponding RKHS on $[0,1]$, since it can be extended to $\mathbb{R}$ by multiplication by an infinitely differentiable function with compact support and then \citep[Theorem 10.21]{wendland2004scattered} can be applied. Note that similar arguments are provided in \citep{bachoc2017gaussian,LopezLopera2017FiniteGPlinear}.
From \citep[Lemma 4.12]{bachoc2022sequential}, the multi-affine extension of the restriction of this constant function to $F$ also satisfies the constraints.
Also, recall that $\HF= {\cal H}_{|F}$.
Hence $ \HF \cap C_F \neq \emptyset$.
Similarly, when $C$ is the set of non-decreasing functions, Condition~\ref{cond:HfintCf} holds for all Mat\'ern kernels, as shown by considering the identity function on $[0,1]$. Finally, when $C$ is the set of convex functions, Condition~\ref{cond:HfintCf} holds for all Mat\'ern kernels, as shown by considering the function $[0,1] \ni x \mapsto x^2$. 

We consider the constrained optimal smoothing: 
\begin{align}\label{eq:constrainedSpline}
	\widehat{u}_{F} &= \underset{v \in \HF \cap C_F}{\arg \min} J_{F}(v),
\end{align}
where, for $v\in \HF$,  
\begin{align}
	\label{eq:JF}
	J_F(v) &=  \|v\|_{\HF}^2 + \frac{1}{\tau} \sum_{i=1}^{n} ((Pv)(x_i)
	- y_i )^2.
\end{align}

The problem in \eqref{eq:constrainedSpline} is a restriction of the problem in \eqref{constrained} announced in Section~\ref{sec:intro}, taking into account that the closure set $F$ is allowed to be different from $[0,1]$. Next, we introduce the standard notion of strong convexity which we will use in this paper to show the existence and unicity  of $\widehat{u}_F$ and also to obtain our error bounds.

\begin{definition} \label{def:strongConvexity}
	For a Hilbert space $V$ and a function $g : V \to \mathbb{R}$, we say that $g$ is strongly convex with parameter $m$ if and only if for all $u,v \in V$ and $t \in [0,1]$, we have
	$$
	g(tu+(1-t)v)\leq tg(u)+(1-t)g(v) -\frac{m}{2}t(1-t)\| u-v\|_V^2.
	$$
\end{definition}

\begin{proposition} \label{prop:spline_strong_convex}
	The function $v \in \HF \mapsto  J_{F}(v)$ is strongly convex with parameter $m = 2$.
\end{proposition}
\begin{proof}
	It is well-known that the function $v \mapsto  \|v\|_{\HF}^2$ is strongly convex with parameter $m = 2$. Furthermore, the function
	$$
	v \mapsto  \frac{1}{\tau} \sum_{i=1}^{n} (P(v(x_i)) - y_i)^2, 
	$$
	is convex as the composition of the affine function ${v \mapsto \frac{1}{\sqrt{\tau}} (P v(x_i) - y_i)_{i=1,\ldots,n}}$ by the squared Euclidean norm which is convex. Hence, $J_F$, which is the sum of the two functions, is strongly convex with parameter $m = 2$. 
\end{proof}

The next lemma addresses the optimization space $\HF \cap C_F$ in \eqref{eq:constrainedSpline}. To prove this lemma and throughout this paper, we will exploit the reproducing property \citep{berlinet2011reproducing} in the RKHSs ${\cal H}$ and $\HF$: for any $f \in {\cal H}$ and $x \in [0,1]$, we have $f(x) = \langle f , K(x , \cdot) \rangle_{\cal H}$. Similarly for $f \in \HF$ and $x \in F$, we have $f(x) = \langle f , K(x , \cdot) \rangle_{\HF}$. 

\begin{lemma} \label{lemma:CFClosed}
	The set $\HF \cap C_F$ is a closed subset of $\HF$ w.r.t. $||\cdot||_{\HF}$.
\end{lemma}

\begin{proof}
	We have $	  \HF \cap C_F= \{v \in \HF:   Pv \in C\}$. Let us fix $v \in \HF$ and consider a sequence $v_n \in \HF \cap C_F$ such that $v_n \to v$ for $||\cdot||_{\HF}$. By the reproducing property, 
	\[
	||v_n - v||_{\infty} \leq ||v_n - v||_{\HF} \sup_{t \in F} \sqrt{K(t,t)}.
	\]
	As $K$ is continuous, we deduce that $||v_n - v||_{\infty} \to 0$. Now $P$ is 1-Lipschitz with respect to  $||\cdot||_{\infty}$ and thus continuous. Hence $P v_n \to Pv$ for $||\cdot||_{\infty}$. As $C$ is a closed subset of $E$, then $Pv \in C$. This proves that $v \in C_F$ and thus    $v \in \HF \cap C_F$.
\end{proof}

%%%%%%%%%%%%%%%%%%%%%%%%%%%%%%%%%%%%%%%%%%%%%%%%%%%%%%%%%%%%%%%%%%%%%%%%%%%%%%%% THEO 1
We now state the existence and unicity  of $\widehat{u}_{F}$, and additionally quantify its regularity. 
\begin{theorem}\label{theo:spline_strong_convex}
	The constrained optimal smoothing problem in \eqref{eq:constrainedSpline} has a unique solution $\widehat{u}_{F}$. Moreover $\widehat{u}_{F}$ is $\frac{\beta}{2}$-H\"older continuous with constant $\displaystyle\sqrt{2c_K}\| \widehat{u}_{F}\|_{\HF} $, recalling $c_K$ from Condition~\ref{cond:contKernel}.
\end{theorem} 
%%%%%%%%%%%%%%%%%%%%%%%%%%%%%%%%%%%%%%%%%%%%%%%%%%%%%%%%%%%%%%%%%%%%%%%%%%%%%%%%%%%%%%%%%%%%%%
\begin{proof}
	As $J_F(v) \geq ||v||_{H_{F}}^2$, we have $\lim_{\| v\| \mapsto\infty} J_F(v)=\infty $. Furthermore, as $C_F$ is a convex set of $E_F$, as $ \HF \cap C_F \neq \emptyset$   (Condition~\ref{cond:HfintCf})   and from Lemma~\ref{lemma:CFClosed},  $\HF \cap C_F  $ is a non empty closed convex set of $E_F$. By Proposition~\ref{prop:spline_strong_convex}, $J_F$ is strictly convex. Then, as $J_F$ is clearly continuous on $\HF$, by %Theorem 1.18 (p 16) in 
	(\citep{hiriart}, Theorem 1.18), \eqref{eq:constrainedSpline} has a unique solution and, as $K$  is  $\beta$-H\"older continuous, we have:
	\begin{align*}
		| \widehat{u}_{F}(t)-\widehat{u}_{F}(s) | 
		&= | \langle \widehat{u}_{F}(\cdot),K(\cdot,t)-K(\cdot,s) \rangle_{\HF} |
		\\
		&\leq \| \widehat{u}_{F}\|_{\HF} \|  K(\cdot,t)-K(\cdot,s)\|_{\HF}
		\\
		&\leq \| \widehat{u}_{F}\|_{\HF}  \left(  K(t,t)-K(t,s)+K(s,s)-K(s,t) \right)^{1/2}
		\\
		&\leq \| \widehat{u}_{F}\|_{\HF}  \displaystyle\sqrt{2c_K} | s-t| ^{\beta/2}.
	\end{align*}
\end{proof}

%%%%%%%%%%%%%%%%%%%%%%%%%%%%%%%%
\section{The discrete constrained optimal smoothing problem} \label{section:discrete:constrained:optimal:smoothing}
%%%%%%%%%%%%%%%%%%%%%%%%%%%%%%%%%%%%%%%%%%%%%%%%%%%%%

We now introduce the finite-dimensional spaces where the approximate to $\widehat{u}_F$ is sought. First, we deal with the notion of grid size compatible with the closure set $F$ of the knots discussed in Sections~\ref{sec:intro} and~\ref{sec:cosp}. Then, we introduce the hat basis functions given by the knots and yielding the piecewise linear approximation   $\widehat{u}_{N,F}$ discussed in Section~\ref{sec:intro}.

We consider a sequence of nested subdivisions, i.e. sets of knots $S_{N} \subset S_{N+1}$ with $S_N: \quad 0=t_{1} <  \cdots < t_{N}= 1$. As explained in Section~\ref{sec:cosp}, $F$ is the closure of $ \underset{N \geq 1}{\bigcup} S_N$ , i.e.
\begin{equation*}
	F = \overline{\underset{N \geq 1}{\bigcup} S_N}.
\end{equation*}
Hence $F$ is a compact set of $[0,1]$ containing $0$ and $1$. The nodes of $S_N$   allow the construction of the finite-dimensional approximation spaces, as we will detail below.

Before tackling this approximation space, we first define the grid size of $S_N$. As the set $F$ is allowed to have holes, it is insufficient to define the grid size simply as $\max_{i=1}^{N-1} | t_{i+1}-t_i |$. Indeed, when $F \neq [0,1]$, this quantity will not tend to zero as $N \to \infty$. To overcome this issue, we need a more nuanced definition of the grid size. For $t \in F$, we define two nodes $t_{S_N}^- \in S_N $ and $t_{S_N}^+ \in S_N$, which are, respectively, the closest grid points of $S_N$ to $t$ on its left and right. In other words, if $t \in S_N$, then $t_{S_N}^- = t_{S_N}^+ = t$, otherwise
\begin{align}\label{tSN}
	t_{S_N}^-   =  \max \{ x, x \in S_N, x \leq t\}, \qquad
	t_{S_N}^{+} =  \min \{ x, x \in S_N, x \geq t\}.
\end{align}
Thus, we can define the grid size as  
\begin{equation}\label{delta_N}
	\delta_N = \sup_{ t \in F}  \min\left(| t-t_{S_N}^- |, | t-t_{S_N}^+ | \right).
\end{equation}
Note that since $F$ is the closure of $S_N$, we have $\delta_N \to 0$ as $N \to \infty$, which is the benefit of the definition in~\eqref{delta_N}. 

We now proceed to formally define the piecewise linear approximation.
We define the set of hat functions  $(\phi_{1},\ldots,\phi_{N})$, similarly as in \citep{bachoc2022sequential},
\begin{equation*}
	\phi_i(t) = \begin{cases}
		\dfrac{t - t_{i-1}}{t_i - t_{i-1}} & \text{if } i \geq 1 \text{ and } t \in [t_{i-1}, t_i], \\ 
		\dfrac{t_{i+1} - t}{t_{i+1} - t_{i}} & \text{if } i \leq N
		-1 \text{ and } t \in [t_{i}, t_{i+1}], \\ 
		0 & \text{otherwise},
	\end{cases}
\end{equation*}
with the convention that $t_0 = -1$ and $t_{N+1} = 2$. Note that a linear combination of the $\phi_i$'s is a piecewise linear function. 

Then, as in \citep{bay2016generalization,bay2017new,GBM}, we place an RKHS structure on the linear combinations of the $\phi_i$'s, by defining a kernel $K_N$. This kernel can also be seen as an approximation of $K$, thus corresponding to a finite-dimensional approximation of a GP with kernel $K$. We let $\HNF$ be the finite-dimensional subspace of $E_F$ defined by  
\begin{align*}
	\HNF = \operatorname{Span}\{ \phi_{i |F}, i = 1 , \ldots, N\}.
\end{align*}
Here, for a function $g: [0,1] \to \mathbb{R}$, we let $g_{|F}$ be the restriction of $g$ to $F$, that is the function $F \to \mathbb{R}$ defined by, for $x \in F$, $g_{|F}(x) = g(x)$.
On $\Hilb_{N,F}$, we now construct the kernel $K_N$ to obtain an RKHS. For this we first need to define  the matrix 
$
\Gamma_N = (K(t_i,t_j))_{1\leq i,j\leq N},
$
and we assume Condition~\ref{cond:kernelStricPos}.
\begin{condition} \label{cond:kernelStricPos}
	$\Gamma_N$ is invertible.
\end{condition}
This condition is verified for a strictly positive definite kernel, for instance the squared exponential kernel (see e.g., \citep{karlin1966t}, Chapter 1, Section 3, Example 5).
\medskip

Then, let us define the following inner product on $\HNF$ :
\begin{equation}\label{scaN}
	\langle u, v \rangle_{N} = c_{u}^\top \Gamma_N^{-1} c_{v},
\end{equation}
where $c_{u} = [u(t_1), \ldots, u(t_N)]^\top$ and $c_{v} = [v(t_1), \ldots, v(t_N)]^\top$. % and $\langle \cdot, \cdot \rangle_{N,F} = \langle \cdot, \cdot \rangle_{N}$.\\
Write $\| \cdot \|_N$ for the corresponding norm. We can now state the RKHS structure of $\HNF$.
\begin{proposition}\label{prop3}[\citep{bay2017new}, Theorem 1]
	The space  $\HNF$, equipped with the scalar product in \eqref{scaN}, is an RKHS with kernel $K_N$ given by
	\begin{equation*}
		\forall x, x' \in F, \qquad K_N(x, x')= \sum_{i,j=1}^N K(t_i,t_{j})\phi_{i}(x) \phi_{j}(x').
	\end{equation*}
\end{proposition}
\begin{proof}
	See the proof in (\citep{bay2017new}, Theorem 1). %,  or in  the last section. 
\end{proof}

These tools allow us to define the discrete problem which will approximate the constrained optimal smoothing problem in \eqref{eq:constrainedSpline}.
We can approximate~\eqref{eq:JF} by the following functional,
%for $u_N \in \HN$,
%\begin{align*}
	%   J_{N}(u_{N}) =\|u_N\|_{N}^2 +   \frac{1}{\tau} \sum_{i=1}^{n} (u_N(x_i) - y_i)^2.
	%\end{align*}
%and
for   $u_N \in \HNF$,  
\begin{align}	 
	\label{eq:JNF}
	J_{N,F}(u_{N}) &=\|u_N\|_{N}^2+  \frac{1}{\tau} \sum_{i=1}^{n} (P(u_N)(x_i) - y_i)^2.
\end{align}
Hence we approximate $\widehat{u}_{F}$ by $\widehat{u}_{N,F}$ solution of 
\begin{align}\label{optN}
	\widehat{u}_{N,F} &= \underset{v_N \in \HNF \cap C_F}{\arg \min} J_{N,F}(v_N). 
\end{align}
Note that when $F = [0,1]$, $\widehat{u}_{N,F}$ corresponds to the approximant of $\widehat{u}_N$. % that was discussed in Section~\ref{sec:intro}.

\begin{proposition} \label{prop:JNF:Frechet}
	$J_{N,F}$ is Fr\'echet differentiable, strongly convex with parameter 2,  and   
	\begin{equation*}
		\lim\limits_{\|v_{N}\|_{N} \to \infty} J_{N,F}(v_{N}) = \infty.
	\end{equation*}
\end{proposition}
\begin{proof}
	It can be easily deduced that $J_N$ is Fréchet differentiable. With the same arguments as in the proof of Proposition~\ref{prop:spline_strong_convex},
	$J_{N,F}$ is strongly convex with parameter 2. 
	Finally, as $J_{N,F}(v_{N}) \geq \|v_N\|_{N}^2$ then $\lim_{\|v_{N}\|_{N} \to \infty} J_{N,F}(v_{N}) = \infty.$
\end{proof}

Then, we define  $\pi_N$ as the piecewise affine interpolation  associated to the subdivision $S_N$,  defined from $E_F$  onto $\HNF$ by
\begin{equation*}
	\forall f\in E_F, \quad \pi_N (f)=\sum\limits_{j=1}^{N} f(t_{j})\phi_{j |F}.
\end{equation*}
Notice that $\pi_N$ is a projection in the sense that $\pi_N \circ \pi_N = \pi_N$. The following proposition provides a first approximation property for this projection. However, it does not offer a quantitative measure of the approximation's quality. This will be addressed in Section~\ref{section:quantitative:properties}. 

\begin{proposition}\label{stability} 
	For any $f\in \HF$, we have $
	\|\pi_N(f)\|_{N}\leq \|f\|_{\HF}
	$, and 
	\begin{equation*}\label{PN}
		\pi_N(f) \xrightarrow[N\to \infty]{} f \quad \mbox{in $E_F$}.
	\end{equation*}  
	Moreover $\HF$ is characterized by  
	\begin{equation*}\label{newdefH}
		\HF = \Big\{f\in E_F~: \  \sup_{N} \|\pi_N(f)\|_{N}<\infty \Big\},
	\end{equation*}
	and, for all $f\in \HF$, by 
	$%\begin{equation*}%\label{normH}
		\|f\|_{\HF}^2= \lim_{N\to \infty}\|\pi_N(f)\|_{N}^2. 
		$%\end{equation*}
\end{proposition}
\begin{proof}
	See (\citep{bay2016generalization}, Theorem 3.1).
\end{proof}
In Proposition~\ref{stability}, and throughout the paper, recall that the convergence in $E_F$ is defined with the uniform norm.
We make the following assumption: 
\begin{condition} \label{constraints:stable:projection}
	$\pi_N(C_F)\subset C_F$.
\end{condition}
This assumption holds for boundedness, monotonicity and convexity constraints (\citep{bachoc2022sequential}, after Condition 4.11). Finally, we can state the existence, unicity and the bound of $\widehat{u}_{N,F}$.

\begin{theorem}\label{theorem:cont:prob:unique:sol}
	Under Conditions~\ref{cond:contKernel} to \ref{constraints:stable:projection}, \eqref{optN} has a unique solution.
\end{theorem} 
\begin{proof}
	From Condition~\ref{cond:HfintCf}, we can take $g\in \HF\cap C_F $. Then, thanks to Condition~\ref{constraints:stable:projection}, $\pi_N(g) \in  \HNF \cap C_F$, so that $\HNF \cap C_F$ is  nonempty. From the same arguments as in Lemma~\ref{lemma:CFClosed}, it is a closed convex subset of $\HNF$. Similarly as in the proof of Theorem~\ref{theo:spline_strong_convex}, we have the conclusion. 
\end{proof}

%%%%%%%%%%%%%%%%%%%%%%%%%%%%%%%%%%%%%%%%%%%%%%
\begin{proposition}\label{d0}
	$\widehat{u}_{N,F}$ is bounded in $\HNF$ independently of $N$, 
	\begin{equation*}
		\|\widehat{u}_{N,F}\|_{N} \leq d_0, 
	\end{equation*}
	where 
	\begin{equation*}
		d_0^2 :=   \|\widehat{u}_{F}\|_{\HF}^2+\frac{1}{\tau} \sum_{i=1}^{n} \left(
		\left( \sup_{t \in [0,1]} \sqrt{K(t,t)} \right)
		\|\widehat{u}_{F}\|_{\HF} +  \vert y_i\vert \right)^2 . 
	\end{equation*}
\end{proposition}
%%%%%%%%%%%%%%%%%%%%%%%%%%%%%%%%%%%%%%%%%%%%%%%%%%%
\begin{proof}
	As $\pi_N(\widehat{u}_{F}) \in \HNF \cap C_F$ and 
	$\|\pi_N(\widehat{u}_{F})\|_{N}\leq \|\widehat{u}_{F}\|_{\HF}$ (from Proposition~\ref{stability}),
	we have the following bounds
	\begin{eqnarray*}
		\|\widehat{u}_{N,F}\|_{N}^2 &\leq& 	J_{N,F}(\widehat{u}_{N,F}) \leq J_{N,F}(\pi_N(\widehat{u}_{F})) \\
		&\leq &  \|\widehat{u}_{F}\|_{\HF}^2 + \frac{1}{\tau} \sum_{i=1}^{n} (P(\pi_N(\widehat{u}_{F}))(x_i) - y_i)^2 \\
		&\leq &  \|\widehat{u}_{F}\|_{\HF}^2 + \frac{1}{\tau} \sum_{i=1}^{n} \left( \left( \sup_{t \in [0,1]} \sqrt{K(t,t)} \right)\|\widehat{u}_{F}\|_{\HF} +\vert y_i\vert \right)^2. 
	\end{eqnarray*}
	The last inequality holds because
	\[
	\left| P(\pi_N(\widehat{u}_{F}))(x_i) \right|
	\le 
	\| 
	\pi_N(\widehat{u}_{F})
	\|_{\infty} 
	\le
	\| 
	\widehat{u}_{F}
	\|_{\infty} 
	\le
	\| 
	\widehat{u}_{F}
	\|_{\HF} 
	\left( \sup_{t \in [0,1]} \sqrt{K(t,t)} \right).
	\] 
\end{proof}

%%%%%%%%%%%%%%%%%%%%%%%%%%%%%%%%%%%%%%%%%%%%%%	
%%%%%%%%%%%%%%%%%%%%%%%%%%%%%%%%%%%%%%%%%%%%%%%%%%
\section{Quantitative properties of the set of approximants $\HNF$}
\label{section:quantitative:properties}
%%%%%%%%%%%%%%%%%%%%%%%%%%%%%%%%%%%%%%%%%%%%%%%%%%%%%
%%%%%%%%%%%%%%%%%%%%%%%%%%%%%%%%%%%%%%%%%%%%%%
In this section, we aim to quantitatively assess whether the class of approximants we have chosen is suitable. The first indicator is based on the error resulting from the finite-dimensional approximation:
\begin{equation}\label{F_Nf}  
	F_N(f)=\|\pi_N(f) -f\|_{\infty}, 
\end{equation}
for $f$ in  $E_F$ or $\HF$. We also write 
\begin{equation}\label{F_Nhat}  
	\widehat{F}_N=F_N(\widehat{u}_{F}). 
\end{equation}

As $\HNF$ is intended to approximate $\HF$, a second indicator evaluating the quality of the RKHS approximation plays a key role in the error bound. One way to achieve this is through their kernels:
\begin{equation}\label{GN}
	G_N = \sup_{t \in F} \| \rho_N(K_N(\cdot,t))-K(\cdot,t) \|_{\HF}^2. 
\end{equation}
Here, since $K$ and $K_N$ belong to different spaces, we extend $K_N$ using the operator $\rho_N$. % and then compare $\rho_N(K_N)$ with $K$
Thus, $\rho_N$ is the extension operator from the approximating space $\HNF$ to the infinite-dimensional space $\HF$ defined as follows: 
\begin{equation}\label{eq:QN}
	\forall v_N \in \HNF, \quad \rho_N(v_N):=\displaystyle\sum_{i=1}^{N} \lambda_i K(\cdot,t_i),
\end{equation}
where $\Lambda = (\lambda_1,\ldots,\lambda_N)^{\top}$ solves $\Gamma_N\Lambda=c_{v_N}$, recalling the definition of $c_{v_N}$ after \eqref{scaN}. The vector $\Lambda$ is defined such that  the operator $\rho_N$ is  an isometry between $\HNF$ and $\HF$, i.e., $\forall v_N \in \HNF$, we can check that
\begin{equation}\label{isorhoN}
	\| \rho_N(v_N)\|^2_{\HF}=\| v_N\|^2_{N}.
\end{equation} 

\begin{remark}\label{remrhoN}
	It is possible to define $\rho_N$ on $E_F$. In fact, $\rho_N(f)$ is denoted by $I_N (f)$ in \citep{karvonen}, and is often called the kernel interpolant because it corresponds to the unique function in the span of $K(\cdot,t_i)$ that interpolates $f$ at the nodes $t_i$, where $t_i \in S_N$:
	\begin{equation*}
		\rho_N(v_N):= \mbox{argmin}_{h \in \HF} \{ \Vert h\Vert_{\HF}, h(t_i)=v_N(t_i), i=1,\ldots,N \}. 
	\end{equation*}
	The difference is that \citep{karvonen} considers interpolation at the observation points, whereas here, it is at the knots.
\end{remark}

We now study both quality indicators in \eqref{F_Nf} and~\eqref{GN}. For $F_N$, it is useful to give the simplest explicit formula to evaluate $\pi_N(f)(t)$, for  $t \in F$ and $f\in E_F$. For this, we  define $w_{{N}_-}(t)$ and $w_{{N}_+}(t)$ as follows.
Recall the definition of $t_{S_N}^{-}$ and $t_{S_N}^{+}$ in \eqref{tSN}. 
If $t \in S_N$, then $w_{{N}_-}(t) = w_{{N}_+}(t) = 1/2$, otherwise
\begin{align*}\label{omegaN}
	w_{{N}_-}(t) =\dfrac{t_{S_N}^{+}-t}{t_{S_N}^{+} -t_{S_N}^-}, 
	\qquad
	w_{{N}_+}(t) =\dfrac{t-t_{S_N}^-}{t_{S_N}^{+} - t_{S_N}^-}.
\end{align*}
This yields
\begin{equation}\label{evalpiN}
	\pi_N(f)(t) = f(t_{S_N}^-)w_{{N}_-}(t)+ f( t_{S_N}^{+}) w_{{N}_+}(t). 
\end{equation}
We then have the following proposition. 
%%%%%%%%%%%%%%%%%%%%%%%%%%%%%%%%%%%%%%%%%%%%%%%%%%%%%%%%%%%%%%%%%%% PROP 7
\begin{proposition}\label{interpolation}
	%If $f \in \HF$ is continuous on $F$ then 
	%\begin{equation}\label{esti1piN}
		%   F_N(f) \leq \omega(f,\delta_N), 
		%\end{equation}
	%where $\omega(f,\delta_N)=\sup_{ \| r-s\| \leq \delta_N  } | f(r)-f(s)| $ is the modulus of continuity of $f$. \\
	Recall $\Psi_f$ and $\delta_N$ from \eqref{eq:psif} and \eqref{delta_N}, respectively. If $f $ is in $ E$  or $E_F$, then 
	\begin{equation}\label{esti1piN}
		F_N(f) \leq 2  \Psi_f(\delta_N).
	\end{equation}
	If $f$ is $\beta$-H\"older continuous 
	\begin{equation}\label{esti2FN} 
		F_N(f)   \leq 2 c_f \delta_N^{\beta}.
	\end{equation} 
	We have, with $c_K$ as in Condition~\ref{cond:contKernel}, 
	\begin{equation}\label{estihatFN} 
		\widehat{F}_N    \leq 2 \sqrt{2c_K} \ \| \widehat{u}_{F}\|_{\HF}  \ \delta_N^{\frac{\beta}{2}}.
	\end{equation} 
	
\end{proposition}
%%%%%%%%%%%%%%%%%%%%%%%%%%%%%%%%%%%%%%%%%%%%%%%%%%%%%%%%%%%%%%%%%%%%%%%%%%%%

\begin{proof}
	As  $w_{{N}_-}(t) +w_{{N}_+}(t)=1$, \eqref{evalpiN} implies 
	\begin{equation*}
		\pi_N(f)(t)-f(t)=(f(t_{S_N}^-)-f(t))w_{{N}_-}(t)+ (f( t_{S_N}^{+})-f(t)) w_{{N}_+}(t). 
	\end{equation*}
	Let us suppose without loss of generality that $t-t_{S_N}^{-} \leq t_{S_N}^{+}-t$. As $t-t_{S_N}^{-} \leq \delta_N$, 
	\begin{equation}
		\label{tin}
		t_{S_N}^{+}-t    = \dfrac{( t_{S_N}^{+}-t)(t-t_{S_N}^{-})}{t-t_{S_N}^{-}} \leq \dfrac{( t_{S_N}^{+}-t) }{t-t_{S_N}^{-}} \delta_N  \leq x_N\delta_N,  
	\end{equation}
	where $x_N=\frac{t_{S_N}^{+}-t_{S_N}^{-}}{t-t_{S_N}^{-}}\geq 1$, so that 
	\begin{align*}
		|\pi_N(f)(t)-f(t)| 
		&\leq  | f(t_{S_N}^-)-f(t)| \dfrac{t_{S_N}^{+}-t}{t_{S_N}^{+}-t_{S_N}^{-}}+ | f( t_{S_N}^{+})-f(t)| \dfrac{t-t_{S_N}^{-}}{t_{S_N}^{+}-t_{S_N}^{-}}
		\\
		&\leq  | f(t_{S_N}^-)-f(t)|  + | f( t_{S_N}^{+})-f(t)| \dfrac{t-t_{S_N}^{-}}{t_{S_N}^{+}-t_{S_N}^{-}}
		\\
		&\leq  M_f(\delta_N)  +\dfrac{| f( t_{S_N}^{+})-f(t)|}{\dfrac{t_{S_N}^{+}-t_{S_N}^{-}}{t-t_{S_N}^{-}}} 
		~ ~ ~ ~
		\mbox{(recall $M_f(\delta)$ from \eqref{eq:Mf:delta})}
		\\
		&\leq  M_f(\delta_N)  +\dfrac{  M_f( t_{S_N}^{+}-t)}{\dfrac{t_{S_N}^{+}-t_{S_N}^{-}}{t-t_{S_N}^{-}}} 
		\leq M_f(\delta_N)  +\dfrac{  M_f(\delta_N x_N)  }{x_N},
		%\\
		%& \leq     M_f(\delta_N)  +\Psi_f( \delta_N),
		%\\ 
		%&\leq    2\Psi_f( \delta_N).
	\end{align*}
	where we can conclude that $|\pi_N(f)(t)-f(t)| \leq     M_f(\delta_N)  +\Psi_f( \delta_N) \leq 2\Psi_f( \delta_N)$. 
	Hence \eqref{esti1piN} holds.
	\eqref{esti2FN} comes from Proposition~\ref{prop:psif} and \eqref{estihatFN} comes from Theorem~\ref{theo:spline_strong_convex}.
\end{proof}

\medskip

For $G_N$, we have the following proposition.

%%%%%%%%%%%%%%%%%%%%%%%%%%%%%%%%%%%%%%%%%%%%%%%%%%%%%%%%%%%%%%%%%%%%%
%%%%%%%%%%%%%%%%%%%%%%%%%%%%%%%%%%%%%%%%%%%%%%%%%%%%%%%%%%%%%%%%
\begin{proposition}\label{convnoyaurho}
	We have
	\begin{equation} \label{eq:GN:to:zero} 
		G_N \xrightarrow[N\to \infty]{} 0. 
	\end{equation}
	Furthermore, if $K$ satisfies Condition~\ref{cond:contKernel}, then
	\begin{equation}\label{d2} 
		G_N \leq d_2 \delta_N^{\beta},
	\end{equation}
	where $d_2:= 6 c_K$.
\end{proposition} 
%%%%%%%%%%%%%%%%%%%%%%%%%%%%%%%%%%%%%%%%%%%%%%%%%%%%%%%%%%%%
%%%%%%%%%%%%%%%%%%%%%%%%%%%%%%%%%%%%%%%%%%%%%%%%%%%%%%%%%%%%%%%%%%%%%
\begin{proof}
	We have $\rho_N\left(K_N(\cdot,t)\right)=\displaystyle\sum_{i=1}^{N} \lambda_i(t) K(\cdot,t_i)$, where $\Gamma_N\Lambda(t)=c_{K_N(\cdot,t)}$, and
	$$
	K_N(\cdot,t)= \sum_{i=1}^N \Bigg( \sum_{j=1}^N  K(t_i,t_j) \phi_{j |F}(t)\Bigg)\phi_{i |F}.
	$$
	Hence
	$\Gamma_N\Lambda(t)=\Gamma_N \phi(t) $, where $\phi(t)=(\phi_{1|F}(t), \ldots,\phi_{N |F}(t))^\top$. Then, $\Lambda(t)=\phi(t)$ and $$
	\rho_N\left(K_N(\cdot,t)\right)=\displaystyle\sum_{i=1}^{N} \phi_{i |F}(t) K(\cdot,t_i).
	$$  
	Applying the reproducing property of $K$ and $K_N$, we obtain
	\begin{align*}
		\langle K(\cdot,t_i) ,K(\cdot,t)\rangle_{\HF} = K(t,t_i), 
		\qquad
		\langle K_N(\cdot,t) ,K_N(\cdot,t)\rangle_{N} = K_N(t,t).
	\end{align*}
	As $\rho_N$ is isometric (see~\eqref{isorhoN}), we have
	\begin{align*}
		\| \rho_N\left(K_N(\cdot,t)\right)-K(\cdot,t)\|_{\HF}^2
		&=\| \rho_N\left(K_N(\cdot,t)\right)\|_{\HF}^2+\|K(\cdot,t)\|_{\HF}^2
		-2\langle\rho_N\left(K_N(\cdot,t)\right),K(\cdot,t)\rangle_{\HF} \\
		&=\|K_N(\cdot,t)\|_{N}^2 +\|K(\cdot,t)\|_{\HF}^2 - 2\sum_{i=1}^N     {\phi}_{i,F}(t)K(t,t_{i})  \\
		%&=\langle K_N(\cdot,t),K_N(\cdot,t)\rangle_{N}  +\langle K(\cdot,t),K(\cdot,t) \rangle_{\HF} - 2\sum_{i=1}^N {\phi}_{i,F}(t)K(t,t_{i})  \\
		&= K_N(t,t)+K(t,t)-2\sum_{i=1}^N {\phi}_{i,F}(t)K(t,t_{i}).
	\end{align*}
	Let $K_t(\cdot)=K(\cdot,t)$.
	%\begin{eqnarray*}
		%| K(t,t)-\sum_{i=0}^N {\phi}_{i,F}(t)K(t,t_{i})| &=&| K_t(t)-\pi_N K_t(t)| \\
		%&\leq & \| K_t-\pi_N K_t\|_{\infty} \leq 2M_1 \delta_N 
		%\end{eqnarray*}
	%If $K$ is continuous over $F \times F$ then 
	%\begin{eqnarray*}
		%| K(t,t)-\sum_{i=0}^N {\phi}_{i,F}(t)K(t,t_{i})| &=&| K_t(t)-\pi_N     K_t(t)| \\
		%&\leq & \| K_t-\pi_N K_t\|_{\infty} \leq \omega(K_t, \delta_N)\leq  \omega_1(K, \delta_N) \leq \omega_2(K, \delta_N)
		%\end{eqnarray*}
	If $K$ satisfies Condition~\ref{cond:contKernel}, then, according to Propositions~\ref{prop:psif} and \ref{interpolation},
	%\Cref{psibeta,esti1piN},
	\begin{align*}
		\left| K_t(t)-\sum_{i=1}^N {\phi}_{i,F}(t)K_{t_i}(t)\right| = | K_t(t)-\pi_N    (K_t(t))| \leq \| K_t-\pi_N (K_t)\|_{\infty} \leq  2c_K \delta_N^{\beta}.
	\end{align*}
	We have, for $t \in [t_{S_N}^-,t_{S_N}^+]$,
	\begin{equation*}
		K_N(t,t)
		=
		K_{t_{S_N}^-}(t_{S_N}^-)w^2_{{N}_-}(t) + K_{t_{S_N}^+}(t_{S_N}^+)w^2_{{N}_+}(t)+2 K_{t_{S_N}^-}(t_{S_N}^+)w_{{N}_+}(t)w_{{N}_-}(t).
	\end{equation*}
	Also
	\begin{align*}
		K_{t_{S_N}^-}(t)w_{{N}_-}(t) &= K_{t_{S_N}^-}(t)w_{{N}_-}(t) [ w_{{N}_-}(t)+w_{{N}_+}(t)], \\
		K_{t_{S_N}^+}(t)w_{{N}_+}(t) &= K_{t_{S_N}^+}(t)w_{{N}_+}(t) [ w_{{N}_-}(t)+w_{{N}_+}(t)].
	\end{align*}
	Hence,
	\begin{align*}
		\Big| K_N(t,t)-
		&
		\sum_{i=1}^N {\phi}_{i,F}(t)K_{t_i}(t)\Big| 
		= \Big| K_N(t,t)-K_{t_{S_N}^-}(t)w_{{N}_-}(t)- K_{t_{S_N}^{+}}(t) w_{{N}_+}(t) \Big|
		\\
		&\leq  \Big|   K_{t_{S_N}^-}(t_{S_N}^-)-K_{t_{S_N}^-}(t)\Big| w_{{N}_-}^2(t) + \Big|   K_{t_{S_N}^+}(t_{S_N}^+)-K_{t_{S_N}^+}(t)\Big| w_{{N}_+}(t)^2
		\\  
		&+  \left[\Big|   K_{t_{S_N}^-}(t_{S_N}^+)-K_{t_{S_N}^-}(t)\Big| + \Big| K_{t_{S_N}^+}(t_{S_N}^-)-K_{t_{S_N}^+}(t)\Big| \right] w_{{N}_+}(t)w_{{N}_-}(t).
		%&+  |   K(t_{S_N}^+,t_{S_N}^-)-K(t,t_{S_N}^-)| w_{{N}_+}(t)w_{{N}_-}(t)  
		%\\
		%&+ | K(t_{S_N}^-,t_{S_N}^+)-K(t,t_{S_N}^+)| w_{{N}_+}(t)w_{{N}_-}(t).    
	\end{align*}
	
	Let us now suppose, without loss of generality, that $t-t_{S_N}^{-} \leq t_{S_N}^{+}-t$. As $K_{t_{S_N}^+}$ and $K_{t_{S_N}^-}$ are $\beta$-H\"older continuous, we have
	\begin{align*}
		\Big|   K_{t_{S_N}^-}(t_{S_N}^-)-K_{t_{S_N}^-}(t) \Big| w^2_{{N}_-}(t)  \leq c_K \delta_N^{\beta}.
	\end{align*}
	Using \eqref{tin}, as $t_{S_N}^{+}-t   \leq x_N\delta_N,$ where $x_N=\frac{t_{S_N}^{+}-t_{S_N}^{-}}{t-t_{S_N}^{-}}\geq 1$, we obtain 
	\begin{align*}
		\Big|  K_{t_{S_N}^+}(t_{S_N}^+)-K_{t_{S_N}^+}(t)\Big| w^2_{{N}_+}(t) 
		&\leq M_{K_{t_{S_N}^+}} (t_{S_N}^+-t)  \left(         \dfrac{t-t_{S_N}^{-}}{t_{S_N}^{+}-t_{S_N}^{-}} \right)^2 
		\\
		&\leq  \dfrac{  M_{K_{t_{S_N}^+}}(\delta_N x_N)  }{x_N^2} 
		\\
		&\leq \dfrac{  M_{K_{t_{S_N}^+}}(\delta_N x_N)  }{x_N}
		%		\\
		\leq  
		\psi_{K_{t_{S_N}^+}}(\delta_N)  \leq c_K \delta_N^{\beta}, 
		\\ 
		\Big|    K_{t_{S_N}^-}(t_{S_N}^+)-K_{t_{S_N}^-}(t)\Big| w_{{N}_+}(t)w_{{N}_-}(t) 
		&\leq    M_{K_{t_{S_N}^-}} (t_{S_N}^+-t)  \left( \dfrac{t-t_{S_N}^{-}}{t_{S_N}^{+}-t_{S_N}^{-}} \right) 
		\\
		&\leq  \dfrac{  M_{K_{t_{S_N}^-}}(\delta_N x_N)  }{x_N} \leq c_K \delta_N^{\beta}, 
		\\    
		\Big|  K_{t_{S_N}^+}(t_{S_N}^-)-K_{t_{S_N}^+}(t)\Big| w_{{N}_+}(t)w_{{N}_-}(t) &\leq  c_K   \delta_N^{\beta}.           \\
	\end{align*}     
	Finally, 
	$$
	\left| K_N(t,t)-\sum_{i=1}^N {\phi}_{i,F}(t)K_{t_i}(t)\right| 
	\leq 4 c_K \delta_N^{\beta},
	$$
	so that $\| \rho_N\left(K_N(\cdot,t)\right)-K(\cdot,t)\|_{\HF}^2 \leq 6 c_K \delta_N^{\beta}$, which allows to conclude the proof of \eqref{d2}, under Condition~\ref{cond:contKernel}.
	The proof of \eqref{eq:GN:to:zero} uses the same arguments and that $\displaystyle \lim_{\delta \to 0} \Psi_f(\delta)=0$ from Proposition~\ref{prop:psif}.
\end{proof} 

%%%%%%%%%%%%%%%%%%%%%%%%%%%%%%%%%%%%%
\section{Error Bound}
\label{sec:errorBounds}
%%%%%%%%%%%%%%%%%%%%%%%%%%%%%%%%%%%%%%%%%%%
We here analyze the error committed when approximating $\widehat{u}_{F}$, the solution of \eqref{eq:constrainedSpline}, by $\widehat{u}_{N,F}$, the solution of \eqref{optN}. 
We not only prove that $\widehat{u}_{N,F} \to \widehat{u}_{F}$ when $N \to \infty$, but also give an error bound. To estimate the error $\| \widehat{u}_{N,F}-\widehat{u}_{F}\|_{\infty}$, we split it in two terms: the piecewise linear interpolation error $\| \pi_N(\widehat{u}_{F}) -\widehat{u}_{F}\|_{\infty}$, and the distance from the approximate solution to the projection of the exact solution, both on $\HNF$, $\|  \pi_N(\widehat{u}_{F})-\widehat{u}_{N,F} \|_{\infty}$. This splitting is treated rigorously in the following proposition.

%%%%%%%%%%%%%%%%%%%%%%%%%%%%%%%%%%%%%%%%%%%%%%%%%%%%%%%%%%  PROP 9   

\begin{proposition}\label{prop8}
	Under Conditions~\ref{cond:contKernel} to \ref{constraints:stable:projection}, there is a constant $c$ such that 
	\begin{equation}\label{estierror}
		\| \widehat{u}_{N,F}-\widehat{u}_{F}\|_{\infty}    \leq    c \ \|\pi_N (\widehat{u}_{F}) - \widehat{u}_{N,F}\|_{N} + d_1 \delta_N^{\beta/2},
	\end{equation} 
	where $d_1 :=  \sqrt{8c_K}\| \widehat{u}_{F} \|_{\HF}$, with $c_K$ and $\beta$ from Condition~\ref{cond:contKernel}.
\end{proposition}

%%%%%%%%%%%%%%%%%%%%%%%%%%%%%%%%%%%%%%%%%%%%%%%%%%%%%%%%%%  PROP 9   

\begin{proof}
	We have
	\begin{equation*}
		\|\widehat{u}_{N,F}-\widehat{u}_{F}\|_{\infty} \leq \|\widehat{u}_{N,F} - \pi_N (\widehat{u}_{F})\|_{\infty} + \|\pi_N (\widehat{u}_{F}) -  \widehat{u}_{F}\|_{\infty}. 
	\end{equation*}
	As $\|\pi_N (\widehat{u}_{F}) -  \widehat{u}_{F}\|_{\infty}=F_N( \widehat{u}_{F})$, according to~\eqref{estihatFN}, we obtain
	\begin{equation}\label{interrd1}
		\|\pi_N (\widehat{u}_{F}) -  \widehat{u}_{F}\|_{\infty} \leq 2 \displaystyle\sqrt{2c_K}\| \widehat{u}_{F}\|_{\HF} \delta_N^{\beta/2}= d_1 \delta_N^{\beta/2}.
	\end{equation} 
	As $\HNF$ is an Hilbertian space of $E_F$, from Lemma 5.1 in \citep{GMB}, there exists a constant $c$ such that, $\forall h_N \in \HNF$,
	\begin{equation}\label{hilb1} 
		\| h_N \|_{\infty} \leq  c\| h_N \|_{N}.
	\end{equation}
	%Indeed, for $x\in F$,
	%\begin{equation*}
		%    |h_N(x)|=|\langle h_N,K_N(\cdot,x) \rangle_{N}|\leq \|h_N\|_{N} \sqrt{K_N(x,x)},
		%\end{equation*}
	%where $K_N(x,x)=\displaystyle\sum_{i,j=1}^N K(t_{i},t_{j}){\phi}_{i,F}(x)\phi_{j,F}(x)$. Since $\displaystyle\sum_{i,j=1}^N{\phi}_{i,F}(x)\phi_{j,F}(x)=1$, we have
	%\begin{equation*}
		%    0 \leq \sup_{x\in F}K_N(x,x)\leq c,
		%\end{equation*}
	%with 
	%\begin{equation*}\label{c}
		%    c = \max_{x,x'\in F}|K(x,x')|.
		%\end{equation*}
\end{proof}
Similarly as in \eqref{hilb1} and up to increasing $c$, for all $h \in \HF$,
\begin{equation}
	\label{hilb2} 
	\| h \|_{\infty} \leq  c\| h \|_{\HF}.
\end{equation}

%As developed in the proof of~Proposition~\ref{prop8}, the piecewise linear interpolation error $\| \pi_N(\widehat{u}_{F}) -\widehat{u}_{F}\|_{\infty}$ can be addressed using the elements already established in ~Section~\ref{section:quantitative:properties}.
It remains to address the second term $\|\pi_N (\widehat{u}_{F}) - \widehat{u}_{N,F}\|_{N}$, that requires a more delicate treatment, provided in the next propositions. In the following, for the sake of readability, some of the proofs will be presented in Appendix~\ref{sectionproofs}.

We first show that the bound of $\|\pi_N (\widehat{u}_{F}) - \widehat{u}_{N,F}\|_{N}$ relies on the characterization of the strong convexity for a differentiable function and the necessary condition of the first order for its minimum. 

%%%%%%%%%%%%%%%%%%%%%%%%%%%%%%%%%%%%%%%%%%%%%%%%%%%%%%%%%%  PROP 10   
\begin{proposition}\label{majmain}
	Under Conditions~\ref{cond:contKernel} to \ref{constraints:stable:projection}, we have
	\begin{equation*}\label{eq1}
		\|\pi_N (\widehat{u}_{F}) - \widehat{u}_{N,F}\|_{N}^2 \leq  J_{N,F}(\pi_N (\widehat{u}_{F})) - J_{N,F}(\widehat{u}_{N,F}),
	\end{equation*}
	with $J_{N,F}$ as in \eqref{eq:JNF}.
\end{proposition}

%%%%%%%%%%%%%%%%%%%%%%%%%%%%%%%%%%%%%%%%%%%%%%%%%%%%%%%%%%  PROP 10
\begin{proof}
	As $J_{N,F}$ is differentiable from Proposition~\ref{prop:JNF:Frechet} with Fréchet derivative $J_{N,F}'$, then the strong convexity leads to, $\forall v,u \in  \HNF$,  
	\begin{equation*}
		J_{N,F}(v) - J_{N,F}(u) \geq \langle J_{N,F}'(u), v-u\rangle_N + \|v-u\|_{N}^2.
	\end{equation*}
	If $u = \widehat{u}_{N,F}$ and $v = \pi_N (\widehat{u}_{F})$, then 
	\begin{equation*}
		J_{N,F}(\pi_N (\widehat{u}_{F})) - J_{N,F}(\widehat{u}_{N,F}) \geq \langle J_{N,F}'(\widehat{u}_{N,F}), \pi_N (\widehat{u}_{F}) - \widehat{u}_{N,F} \rangle_N + \|\pi_N (\widehat{u}_{F}) - \widehat{u}_{N,F}\|_{N}^2.
	\end{equation*}
	As $\pi_N (\widehat{u}_{F}) \in \HNF\cap C_F$ thanks to Condition~\ref{constraints:stable:projection}, and $\widehat{u}_{N,F}$ solves \eqref{optN}, we have $$\langle J_{N,F}'(\widehat{u}_{N,F}), \pi_N (\widehat{u}_{F}) - \widehat{u}_{N,F} \rangle_N \geq 0,$$
	which allows to conclude the proof.
	%so that 
	%$$   \|\pi_N (\widehat{u}_{F}) - \widehat{u}_{N,F}\|_{N}^2 \leq  J_{N,F}(\pi_N (\widehat{u}_{F})) - J_{N,F}(\widehat{u}_{N,F}).$$ 	
\end{proof}

The next proposition derives the bounds of $J_{N,F}(\pi_N (\widehat{u}_{F}))$ and $J_{N,F}(\widehat{u}_{N,F})$.
%%%%%%%%%%%%%%%%%%%%%%%%%%%%%%%%%%%%%%%%%%%%%%%%%%%%%%%%%%  PROP 11   
\begin{proposition}\label{intermed1}
	Under Conditions~\ref{cond:contKernel} to \ref{constraints:stable:projection}, we have
	\begin{align}\label{inter1}
		J_{N,F}(\pi_N (\widehat{u}_{F})) &= -\ENhat +J_F(\widehat{u}_F)+\epsilon_N,
		\\
		J_{N,F}(\widehat{u}_{N,F}) &= J_{F}( \rho_N(\widehat{u}_{N,F}))+\eta_N,
		\label{inter2}
	\end{align}
	where
	$
	\ENhat=\|\widehat{u}_{F}\|_{\HF}^2-\|\pi_N(\widehat{u}_{F})\|_{N}^2,  
	$
	\begin{align*}\label{epsiN}
		| \epsilon_N | \leq d_3  \delta_N^{\beta/2}, \qquad d_3 &=\frac{2nd_1}{\tau}   \Big(c\|\widehat{u}_{F}\|_{\HF} +\max_i | y_i|\Big),\\
		| \eta_N | \leq d_4 \delta_N^{\beta/2}, \qquad d_4 &=\frac{2n\sqrt{d_2}}{\tau}      d_0 \Big(c \ d_0 +\max_i |y_i|\Big). 
	\end{align*}
\end{proposition}
%%%%%%%%%%%%%%%%%%%%%%%%%%%%%%%%%%%%%%%%%%%%%%%%%%%%%%%%%%  PROP 11   
\begin{proof}
	This proposition has no standalone value but it serves as an intermediary result required for subsequent derivations. Its proof relies solely on computations and the application of previously established results. The complete proof is provided in Appendix~\ref{sectionproofs}.
\end{proof}

At this stage, we can prove that $\widehat{u}_{N,F}  \xrightarrow[N\to \infty]{} \widehat{u}_{F}$ in $E_F$. %The error bound requires more effort.  
\begin{theorem}\label{thint}  
	Under Conditions~\ref{cond:contKernel} to \ref{constraints:stable:projection},
	\begin{equation}\label{mainlim}
		\widehat{u}_{N,F}  \xrightarrow[N\to \infty]{} \widehat{u}_{F}   \quad \mbox{in $E_F$}.
	\end{equation} 
\end{theorem}
\begin{proof}
	The convergence in~\eqref{mainlim} is already proved in \citep{GMB}. The proof is also provided in Appendix~\ref{sectionproofs} for a self-contained reading. 
\end{proof}

Now, from Propositions~\ref{majmain} and \ref{intermed1}, we can establish the error bound of $\|\pi_N (\widehat{u}_{F}) - \widehat{u}_{N,F}\|_{N}$.
%%%%%%%%%%%%%%%%%%%%%%%%%%%%%%%%%%%%%%%%%%%%%%%%%%%%%%%%%%%%% PROP 5.5
\begin{proposition}\label{theorem:inter}
	Under Conditions~\ref{cond:contKernel} to \ref{constraints:stable:projection} 
	\begin{equation}\label{esti3piN}
		\|\pi_N (\widehat{u}_{F}) - \widehat{u}_{N,F}\|_{N}^2  \leq   
		J_F(\widehat{u}_F)-J_{F}( \rho_N(\widehat{u}_{N,F}))   +  d_5  \delta_N^{\beta/2},
	\end{equation}
	where $d_5 = d_3 + d_4$.
\end{proposition} 
%%%%%%%%%%%%%%%%%%%%%%%%%%%%%%%%%%%%%%%%%%%%%%% PROP 12 
\begin{proof}
	From Proposition~\ref{stability}, $\ENhat \geq 0 $. Using Propositions~\ref{majmain} and \ref{intermed1},   
	\begin{align*} 
		\|\pi_N (\widehat{u}_{F}) - \widehat{u}_{N,F}\|_{N}^2 
		&\leq  J_{N,F}(\pi_N (\widehat{u}_{F})) - J_{N,F}(\widehat{u}_{N,F})
		\\
		&= -\ENhat+J_F(\widehat{u}_F)+\epsilon_N - J_{F}( \rho_N(\widehat{u}_{N,F})) -\eta_N 
		\\
		&\leq J_F(\widehat{u}_F) - J_{F}( \rho_N(\widehat{u}_{N,F})) + d_3  \delta_N^{\beta/2} + d_4 \delta_N^{\beta/2}.
	\end{align*}
	Since $d_5=d_3 +d_4$, we have the bound in \eqref{esti3piN}.  
\end{proof}
%%%%%%%%%%%%%%%%%%%%%%%%%%%%%%%%%%%%%%%%%%%%%%%
%%%%%%%%%%%%%%%%%%%%%%%%%%%%%%%%%%%%%%%%%%%%%%%%%%%%%%%%%%%%%

To complete the construction of an error bound that is easy to read and interpret, we need to add a third quantity to the first two ($\beta$, which measures regularity, and $\delta_N$, which measures the grid size). This final quantity required for the error bound is the distance, in $\HF$, between  $\rho_N(\widehat{u}_{N,F})$ and the constraints $C_F$. This is the distance of the kernel interpolant of $\widehat{u}_{N,F}$ to the set of constraints:
\begin{eqnarray}\label{alphaN}
	\alpha_N:=d(\rho_N(\widehat{u}_{N,F}), C_F)=\| P_C(\rho_N(\widehat{u}_{N,F}))- \rho_N(\widehat{u}_{N,F})\|_{\HF},
\end{eqnarray}
where $P_C$ in $\HF$ is the orthogonal projection in $\HF$ onto the closed convex set $C_F$. 
\begin{proposition} \label{rem:alphaN}
	$\alpha_N$  tends to zero as $N \to \infty$.
\end{proposition}
\begin{proof}
	Indeed, as  $\widehat{u}_{F}  \in C_F$,
	$$
	\alpha_N = \|\rho_N(\widehat{u}_{N,F})- P_C(\rho_N(\widehat{u}_{N,F}) \|_{\HF}  \leq    \| \rho_N(\widehat{u}_{N,F})- \widehat{u}_{F}   \|_{\HF}.
	$$
	This bound is a coarse one: the aim of this proposition is not to provide an estimation for $\alpha_N$ in the general setting, but only the proof of the convergence to zero.  
	
	According to \eqref{intbound}, we have that
	\begin{equation}\label{convj}
		J_{F}(\rho_N(\widehat{u}_{N,F}) ) \xrightarrow[N \to \infty]{}  J_F(\widehat{u}_F).
	\end{equation}  
	According to \eqref{weakcompact}, 
	to the reproducing property of $\HF$ and by construction of the multi-affine extension $P$ in Definition \ref{def:multi-affine:extension}, there exists a subsequence $P(\rho_{N_k}(\widehat{u}_{N_k,F}))$, such that 
	$$ \lim_{k \to \infty} P(\rho_{N_k}(\widehat{u}_{N_k,F}))(x_i)= P(\widehat{u}_{F})(x_i), \ \forall i=1,\ldots n. $$
	In addition, the sequence   $P(\rho_N(\widehat{u}_{N,F}))(x_i)$ can have a unique accumulation point so that, as it is bounded, $\lim_{N \to \infty} P(\rho_N(\widehat{u}_{N,F}))(x_i)=P(\widehat{u}_{F})(x_i), \ \forall i=1,\ldots n $. Thus,  
	$$
	\displaystyle\frac{1}{\tau} \sum_{i=1}^{n} \left( P(\rho_N(\widehat{u}_{N,F}))(x_i) - y_i \right)^2 \to \displaystyle\frac{1}{\tau} \sum_{i=1}^{n} \left( P(\widehat{u}_F)(x_i) - y_i \right)^2.
	$$
	Thanks to \eqref{convj}, we have 
	$$
	\|   \rho_N(\widehat{u}_{N,F}) \|_{\HF} \to \|  \widehat{u}_{F}   \|_{\HF}.
	$$ 
	This property, combined with the weak convergence, leads to the convergence of $\rho_N(\widehat{u}_{N,F})$: % to $\widehat{u}_{F}$ in $\HF$: 
	$$ 
	\| \rho_N(\widehat{u}_{N,F})- \widehat{u}_{F}   \|_{\HF} \to 0.
	$$
\end{proof} 
\begin{remark}
	The quantity $ \| \rho_N(\widehat{u}_{N,F})- \widehat{u}_{F}   \|_{\HF}$ is far from being a good estimation of $\alpha_N$. In the proof, we bound $\alpha_N$ by this term  because as $\widehat{u}_{N, F}$ may not belong to $\HF$, the term  $ \| \rho_N(\widehat{u}_{N, F})-\widehat{u}_{N,F} \|_{\HF}$ may not be defined. Nevertheless, in the case where $\HNF$ in included in $\HF$, one has the bound 
	$$\alpha_N \leq \| \rho_N(\widehat{u}_{N, F})-\widehat{u}_{N,F} \|_{\HF}.$$
	
	There are many references on the treatment of this term $\| \rho_N(\widehat{u}_{N, F})-\widehat{u}_{N,F} \|_{\HF}$  which quantifies the interpolation accuracy (see, e.g.,~\citep{hangel,narco,wynne}). It is often denoted by $\Vert I_X (f)-f \Vert_{\HF}$ where $I_X=\rho_N$ and $X$ is the set of interpolation points ($S_N$ in our case). In \citep{karvonen},  it is written that, in the general setting, the term  $\Vert f-I_X (f)\Vert_{\HF}$   can tend to zero arbitrarily slowly. Nevertheless, for particular RKHS, for example those corresponding to particular Sobolev spaces, and for functions $f$ smooth enough, this term can be bounded in terms of the fill distance, written $h_X$, which corresponds to the density of locations of $X$. This fill distance corresponds to $\delta_N$ in our case. 
	
	In \citep{narco}, the case where $k(x,y)=\Phi(\vert x - y \vert)$ is considered. This means that the RKHS corresponds to the Native space issued from the radial basis function (RBF) $\Phi$. There, $I_X(f)$ is called the RBF interpolant. For one dimensional functions, if the RKHS corresponds to the classical Sololev space  $W_2^{  \alpha }=\{ f \in L_2, D^{  k}f \in  L_2, k \leq \alpha \}  $, where $D^{  k}f$ denotes the distributional derivative, then one can derive the error estimate (see Lemma 4.1 and Theorem 4.2 in \citep{narco}), for $f\in W_2^{  r }$,  
	\begin{equation*}
		\Vert f-I_X (f)\Vert_{W_2^{  \alpha }} \leq C h_X^{r-\alpha}  \Vert f \Vert_{W_2^{  r }},
	\end{equation*}
	where $\alpha \leq r$. Note that $f$ has to be in a smoother Sobolev space. This is called the doubling trick by~\citep{hangel}. There, an estimate is given to some other Sobolev space (see Theorem 5.1).
	
	When the underlying structure of the data is a stochastic process, the interpolation results are often given in the Gaussian process setting, also known as Kriging~\citep{stein1999interpolation}. For instance in~\citep{wynne}, the interpolant is given by the Gaussian process mean. An error estimate function of the fill distance is provided in their Theorem 1 which . Besides, in their proof of Theorem 1, they use the estimate of \citep{narco}.
\end{remark}

Let us now provide the error bounds.
We first address the case where $\rho_N(\widehat{u}_{N,F})\in C_F$, i.e. $\alpha_N = 0$, which corresponds to Condition~\ref{constraintsfinal}.
\begin{condition} \label{constraintsfinal}
	For $N$ large enough, $\rho_N(\widehat{u}_{N,F})\in C_F.$
\end{condition}

%%%%%%%%%%%%%%%%%%%%%%%%%%%%%%%%%%%%%%%%%%% THEO3 
\begin{theorem}  \label{label:UB:no:extra:assumption}
	Under Conditions~\ref{cond:contKernel} to \ref{constraintsfinal},
	\begin{equation}
		\| \widehat{u}_{N,F}-\widehat{u}_{F}\|_{\infty} = \mathcal{O}( \delta_N^{\beta/4}).
	\end{equation} 
\end{theorem}
%%%%%%%%%%%%%%%%%%%%%%%%%%%%%%%%%%%%%%%%%%% THEO3

\begin{proof}
	From Proposition~\ref{prop8}, $\| \widehat{u}_{N,F}-\widehat{u}_{F}\|_{\infty}    \leq    c \ \|\pi_N (\widehat{u}_{F}) - \widehat{u}_{N,F}\|_{N} + d_1 \delta_N^{\beta/2}.$ 
	If Condition~\ref{constraintsfinal} is verified, then $\rho_N(\widehat{u}_{N,F}) \in C_F$ so that $J_F(\widehat{u}_F) - J_{F}( \rho_N(\widehat{u}_{N,F})) \leq 0$. Hence, from Proposition~\ref{theorem:inter},
	\begin{displaymath}
		\|\pi_N (\widehat{u}_{F}) - \widehat{u}_{N,F}\|_{N}^2  \leq   
		d_5  \delta_N^{\beta/2},
	\end{displaymath} 
	so that 
	\begin{displaymath}
		\| \widehat{u}_{N,F}-\widehat{u}_{F}\|_{\infty}    \leq     c  \sqrt{ d_5  \delta_N^{\beta/2} }   + d_1  \delta_N^{\beta/2}. 
	\end{displaymath} 
\end{proof}

Intuitively, Condition~\ref{constraintsfinal} is expected to hold when $\widehat{u}_{N,F}$ is significantly ``inside'' the constraint set, since $\rho_N(\widehat{u}_{N,F})$ is expected to be close to $\widehat{u}_{N,F}$ for large $N$. 
Nevertheless, when $\widehat{u}_{N,F}$ is close to the boundary of the constraint set, then its kernel interpolant  $\rho_N(\widehat{u}_{N,F})$ could fall outside of this set. To give a very simple example, when the constraint set imposes functions to take values in $[0,1]$, it is possible that the values of a function at the knots are in $[0,1]$ but very close to $0$ or $1$, so that its kernel interpolant function takes some values outside of $[0,1]$, for some kernels $K$.

Hence, it is valuable to analyze the case where Condition~\ref{constraintsfinal} does not hold (i.e. $\alpha_N \ne 0$). In this case,  the convergence proof of the error bound becomes more challenging as we show in~Theorem~\ref{th4} next.
%%%%%%%%%%%%%%%%%%%%%%%%%%%%%%%%%%%%%%%%%%% THEO4
\begin{theorem}\label{th4}  
	Under Conditions~\ref{cond:contKernel} to \ref{constraints:stable:projection},
	\  with $c$ as in Proposition~\ref{prop8}, 
	\begin{equation}\label{mainbound2}
		\| \widehat{u}_{N,F}  - \widehat{u}_{F} \|_{\infty}  \leq  c \  \sqrt{d_8 \alpha_N +  d_5  \delta_N^{\beta/2}} + d_1 \delta_N^{\beta/2},
	\end{equation} 
	where $d_8$ is a constant defined by \eqref{d6}, \eqref{d7} and \eqref{d8}. 
\end{theorem}
%%%%%%%%%%%%%%%%%%%%%%%%%%%%%%%%%%%%%%%%%%% THEO4

\begin{proof} 
	We have
	\begin{align*} 
		J_F(\widehat{u}_F)-J_{F}( \rho_N(\widehat{u}_{N,F}))
		%&=  J_F(\widehat{u}_F)-J_{F}( \rho_N(\widehat{u}_{N,F}))
		%\\
		=& \ J_F(\widehat{u}_F)-J_{F}(P_C(\rho_N(\widehat{u}_{N,F}))) \\ 
		&+ J_{F}(P_C(\rho_N(\widehat{u}_{N,F})))-J_{F}(\rho_N(\widehat{u}_{N,F})).
	\end{align*}
	As $P_C(\rho_N(\widehat{u}_{N,F})) \in \HF \cap C_F$, then  $J_F(\widehat{u}_F)-J_{F}(P_C(\rho_N(\widehat{u}_{N,F})))\leq 0$. This implies that  
	\begin{align*}
		J_F(\widehat{u}_F)-J_{F}( \rho_N(\widehat{u}_{N,F})) 
		& \leq J_{F}(P_C(\rho_N(\widehat{u}_{N,F})))-J_{F}(\rho_N(\widehat{u}_{N,F})) \\
		&= \| P_C(\rho_N(\widehat{u}_{N,F})) \|^2_{\HF}-\|  \rho_N(\widehat{u}_{N,F}) \|^2_{\HF} + R, 
	\end{align*}
	where 
	\begin{align*}
		R =& \ \frac{1}{\tau} \sum_{i=1}^{n} \bigg[ \bigg( (PP_C(\rho_N(\widehat{u}_{N,F})))(x_i) - y_i \bigg)^2 - \bigg( (P\rho_N(\widehat{u}_{N,F}))(x_i)   - y_i\bigg)^2\bigg] \\
		=& \ \frac{1}{\tau} \sum_{i=1}^{n}  \bigg[\bigg( (PP_C(\rho_N(\widehat{u}_{N,F})))(x_i) - (P\rho_N(\widehat{u}_{N,F}))(x_i) \bigg) \\
		& \hspace{7ex} \times \bigg( (PP_C(\rho_N(\widehat{u}_{N,F})))  (x_i) + (P\rho_N(\widehat{u}_{N,F}))(x_i) - 2y_i \bigg)\bigg]. 
	\end{align*}
	As  $P_C$ is 1-Lipschitz, then with Proposition \eqref{d0},
	\begin{align*} 
		\| P_C(\rho_N(\widehat{u}_{N,F}))-P_C(\widehat{u}_{F})\|_{\HF}  
		&\leq \| \rho_N(\widehat{u}_{N,F})-\widehat{u}_{F}\|_{\HF} \\
		&\leq  \|\rho_N(\widehat{u}_{N,F})\|_{\HF} +\| \widehat{u}_{F}\|_{\HF}
		=  \| \widehat{u}_{N,F}\|_{N} +\| \widehat{u}_{F}\|_{\HF} \\
		&\leq d_0+ \| \widehat{u}_{F}\|_{\HF}. 
	\end{align*} 
	According to~\eqref{hilb2},  $\| P_C(\rho_N(\widehat{u}_{N,F}))-P_C(\widehat{u}_{F})\|_{\infty}  \leq c \ (  d_0+ \| \widehat{u}_{F}\|_{\HF})$. Hence 
	\begin{align*}
		\| P_C(\rho_N(\widehat{u}_{N,F})) \|_{\infty}  &\leq c(d_0+ \| \widehat{u}_{F}\|_{\HF})+\|  P_C(\widehat{u}_{F})\|_{\infty} . 
	\end{align*}
	With the definition of the multi-affine extension $P$ (see Definition~\ref{def:multi-affine:extension}) we deduce
	\begin{align*}
		\vert  P(P_C(\rho_N(\widehat{u}_{N,F})))(x_i)\vert \leq  	\| P_C(\rho_N(\widehat{u}_{N,F})) \|_{\infty}  &\leq c(d_0+ \| \widehat{u}_{F}\|_{\HF})+\|  P_C(\widehat{u}_{F})\|_{\infty}.
	\end{align*}
	As we have 
	\begin{align*}
		\vert P\rho_N(\widehat{u}_{N,F})(x_i) \vert \leq  
		\Vert \rho_N(\widehat{u}_{N,F}) \Vert_{\infty} 
		\leq c \Vert \rho_N(\widehat{u}_{N,F}) \Vert_{\HF}  
		=  c \Vert  \widehat{u}_{N,F} \Vert_{N} \leq c \ d_0, 
	\end{align*}
	this implies that, for all $i\in \{1,\ldots,n\}$, 
	\begin{align}\label{d6}
		| (PP_C(\rho_N(\widehat{u}_{N,F})))(x_i)+ (P(\rho_N(\widehat{u}_{N,F})))(x_i) - 2y_i  | \leq d_6,
	\end{align} 
	where 
	$
	d_6:= 2c \ d_0+ c \ \| \widehat{u}_{F}\|_{\HF} +\|  P_C(\widehat{u}_{F})\|_{\infty}  +2 \max \vert y_i \vert
	$. Also, 
	\begin{equation}\label{d7}
		\| P_C(\rho_N(\widehat{u}_{N,F})) \|_{\HF}+\|  \rho_N(\widehat{u}_{N,F}) \|_{\HF} \leq d_7
	\end{equation}
	where 
	$
	d_7:= 2d_0+ \| \widehat{u}_{F}\|_{\HF}+\|  P_C(\widehat{u}_{F})\|_{\HF}
	$. Moreover,
	\begin{align*}
		\Big|  (PP_C(\rho_N(\widehat{u}_{N,F})))(x_i)-P(\rho_N(\widehat{u}_{N,F}))(x_i) \Big|   
		&\leq     \| P_C(\rho_N(\widehat{u}_{N,F})) -\rho_N(\widehat{u}_{N,F}) \|_{\infty} 	\\
		&\leq   c \| P_C(\rho_N(\widehat{u}_{N,F})) -\rho_N(\widehat{u}_{N,F}) \|_{\HF}  
		= c \alpha_N.
	\end{align*}
	Hence $R \leq \dfrac{n \ c d_6}{\tau} \alpha_N$. As 
	\begin{align*} 
		\| P_C\rho_N(\widehat{u}_{N,F}) \|^2_{\HF}-\|  \rho_N(\widehat{u}_{N,F}) \|^2_{\HF}  
		=& \  ( \| P_C\rho_N(\widehat{u}_{N,F}) \|_{\HF}+\|  \rho_N(\widehat{u}_{N,F}) \|_{\HF})  
		\\ 
		& \times  (\| P_C\rho_N(\widehat{u}_{N,F}) \|_{\HF}-\|  \rho_N(\widehat{u}_{N,F}) \|_{\HF}) 
		\leq  d_7  \alpha_N,
	\end{align*}
	then 
	\begin{align}\label{d8} 
		J_F(\widehat{u}_F)-J_{F}( \rho_N(\widehat{u}_{N,F}))  \leq   d_8 \alpha_N,
	\end{align}
	where 
	$
	d_8:= d_7 +\dfrac{ncd_6}{\tau}. 
	$
	Using \eqref{estierror} and \eqref{esti3piN} from Propositions~\ref{prop8} and~\ref{theorem:inter}, we obtain the error bound in \eqref{mainbound2}. 
	%we finally obtain   
	%\begin{equation*} 
		%    \| \widehat{u}_{N,F}-\widehat{u}_{F}\|_{\infty}    \leq    c \  \sqrt{d_8 \alpha_N +  d_5  \delta_N^{\beta/2}} + d_1 \delta_N^{\beta/2}.
		%\end{equation*}
\end{proof}

\begin{remark} \label{rem:multiD}
	With a similar approach as provided here, it is possible to provide error bounds on the numerical approximation of the constrained smoothing problem in higher dimensions. In particular, the multi-affine extension is defined for general dimensions in \citep{bachoc2022sequential}, and all its properties, including those related to the constraint sets of bounded, monotonic, and componentwise convex functions, are also established for general dimensions. Nevertheless, presenting detailed proofs, as we do here, in general dimension yields significantly more complex notations and cumbersome arguments. To maintain readability, we present our results and proofs in one dimension.  
\end{remark}

\begin{remark} \label{rem:regularity}
	Our error bounds depend on the regularity parameter $\beta \in (0,1]$ for the kernel $K$ (Condition~\ref{cond:contKernel}), using the notion of H\"older-continuity, and with the rate $\mathcal{O}(\delta_N^{\beta/4})$ in Theorem~\ref{label:UB:no:extra:assumption}.
	It is natural to ask whether faster decay rates of the upper bounds could be achieved with stronger regularity assumptions, particularly by assuming derivatives of multiple orders. However, it is unclear if this additional regularity would be beneficial in our setting. This is because we rely on piecewise linear interpolation, which typically does not gain further benefits from regularity beyond Lipschitzness. As explained in \citep{bachoc2022sequential}, piecewise linear interpolation is crucial for numerically handling standard constraint sets (boundedness, monotonicity and convexity). Using an interpolation scheme that would benefit from regularity beyond Lipschitzness, for instance piecewise polynomial interpolation, is not suitable for numerically handling these constraint sets.   
\end{remark}

%%%%%%%%%%%%%%%%%%%%%%%%%%%%%%%%%%%%%%%%%%
\section{Numerical Experiments}
\label{sec:numexp} 
%%%%%%%%%%%%%%%%%%%%%%%%%%%%%%%%%%%%%%%%%%%%%%%
In this section, we aim to numerically illustrate Theorems~\ref{label:UB:no:extra:assumption} and~\ref{th4}. Let us recall that the approximate solution $\widehat{u}_{N,F}$ is also the MAP estimate of a GP approximation conditionally to
noisy observations and the inequality constraints \citep{GMB}. Our numerical assessment relies on this property. Therefore, we consider constrained GPs with stationary Mat\'ern kernels~\citep{Genton2001Kernels}:
\begin{equation}
	K(x,x') = \sigma^2 \frac{2^{1-\nu}}{\Gamma(\nu)} \left(\sqrt{2\nu} \frac{|x-x'|}{\ell}\right)^\nu H_\nu\left(\sqrt{2\nu} \frac{|x-x'|}{\ell}\right),
	\label{eq:Matkernel}
\end{equation}
where $x, x' \in [0,1]$, $\Gamma$ is the Gamma function, $H_\nu$ is the modified Bessel function of the second kind of order $\nu$, and $(\sigma^2, \ell) \in (0, \infty)^2$ are the variance and length-scale parameters, respectively. The parameter $\nu \in (0, \infty)$ allows controlling the regularity of the GP. The larger $\nu$, the smoother the GP samples. %As $\nu \to \infty$, the kernel in~\eqref{eq:Matkernel} converges to the SE kernel. 

Given the settings above, we sample twenty constrained GP replicates using the finite-dimensional approximation in~\citep{LopezLopera2017FiniteGPlinear} assuming an equispaced grid of knots with $N = 200$. The choice of $N$ balances the need for better resolution of the piecewise approximation while considering the computational limitations inherent in Monte Carlo techniques. To introduce noise, we corrupt the samples by using independent centered Gaussian noises with $\tau = 5 \times 10^{-2}$. 

We compute the MAP estimate $\widehat{u}_{N, F}$ for each random noisy replicate. This procedure results in twenty predictors that will be used to illustrate Theorems~\ref{label:UB:no:extra:assumption} and~\ref{th4}. For the predictor $\widehat{u}_{F}$, which cannot be evaluated in practice, we approximate it using $\widehat{u}_{N, F}$ assuming an equispaced grid of knots with $N = 10^3$. As parameter estimation is not the focus here, we use the same covariance parameters and noise variance that have been set to generate the constrained GP replicates.

In the first part of our experiments (Section~\ref{sec:numexp:subsec:fixBeta}), we focus on the cases where the grid of the knots is either dense or not. Thus, we fix $\nu = 5/2$ to have the same regularity conditions. In the second part (Section~\ref{sec:numexp:subsec:varyBeta}), we vary $\nu$ seeking to test convergence for different values of $\beta$ while keeping promoting a dense grid of knots. In our context, the link between $\nu$ and $\beta$ is given by $\beta = \min(1, 2\nu)$ (see for instance \citep{loh2015estimating}). 

For dense grids, we conduct the MaxMod algorithm introduced by~\citep{bachoc2022sequential} and discussed in Section~\ref{sec:intro}. We recall that this algorithm iteratively constructs a constrained MAP estimate by sequentially allocating knots for the piecewise linear approximation, each iteration aiming to maximize its modification. We consider a minimal initial number of knots (i.e. $N_0 = 2$), and a maximal budget $N_{\max} = 250$. This budget has been set aiming for a trade-off between computational time and numerical stability due to inversion of covariance matrices. Using the MaxMod algorithm will allow to verify convergence of the error bounds without involving equispaced grid of knots necessarily.

The implementation of the constrained GPs and the MaxMod algorithm are based on the R package \texttt{lineqGPR} \citep{LopezLopera2022LineqGPR}.

\subsection{Error bounds with fixed regularity assumptions}
\label{sec:numexp:subsec:fixBeta}
In this experiment, we sample random GP replicates under monotonicity and boundedness constraints (see Figure~\ref{fig:toyExample1_Samples}). We impose the boundedness constraint $0 \leq Y(x) \leq 1$, for all $x \in [0,1]$, and use a Mat\'ern 5/2 kernel with 
%\begin{equation}
	%    k(x, x') = \sigma^2 \left(1+\sqrt{5}\frac{|x-x'|}{\ell}+ \frac{5}{3} \frac{|x-x'|^2}{\ell^2}\right)\exp\left(-\sqrt{5} \frac{|x-x'|}{\ell}\right),
	%   \label{eq:Mat52kernel}
	%\end{equation}
$\sigma^2 = 1$ and $\ell = 0.4$. We consider two cases where the grid of knots is dense or not. For dense grids, we apply the MaxMod algorithm to each random replicate, resulting in twenty distinct MAP predictors. These predictors are then used in the subsequent experiments to illustrate our theorems. For non-dense grids, we restrict the addition of knots to the interval $I = [0, 0.3] \cup [0.6, 1]$. The refinement process is then conducted via rejection sampling with $t \sim \operatorname{Uniform}(0, 1)$. In both cases, we set $N_{\max} = 250$. 

Figure~\ref{fig:toyExample1_Boxplots} presents boxplots of the error $\|\widehat{u}_{N, F} - \widehat{u}_{F}\|_{\infty}$ and the grid size $\delta_N$ (defined in \eqref{delta_N}) for the twenty replicates. We must remark that the asymptotic error bounds of $\|\widehat{u}_{N, F} - \widehat{u}_{F}\|_{\infty}$ cannot be displayed, as $\alpha_N$ (defined in \eqref{alphaN}) cannot be computed numerically. We observe that the error decreases as $\delta_N$ decreases, which is consistent with Theorems~\ref{label:UB:no:extra:assumption} and~\ref{th4} as the asymptotic error there bounds become smaller as $\delta_N$ decreases. The boxplots show median error values smaller than $10^{-3}$ once the maximal budget $N_{\max}$ is reached, except for the example under boundedness constraints with non-dense grids of knots. This increase in error is due to abrupt changes in the MAP around $x = 0.3$ and $x = 0.6$ (limits of the rejection interval). To achieve smaller error values, it is possible to repeat the experiments with a larger $N_{\max}$ expecting adding knots close to the limits. In particular for this example, a median error value smaller than $10^{-3}$ is achieved after $N = 275$.
\begin{figure}[t!]
	\centering
	\includegraphics[height = 0.285\linewidth]{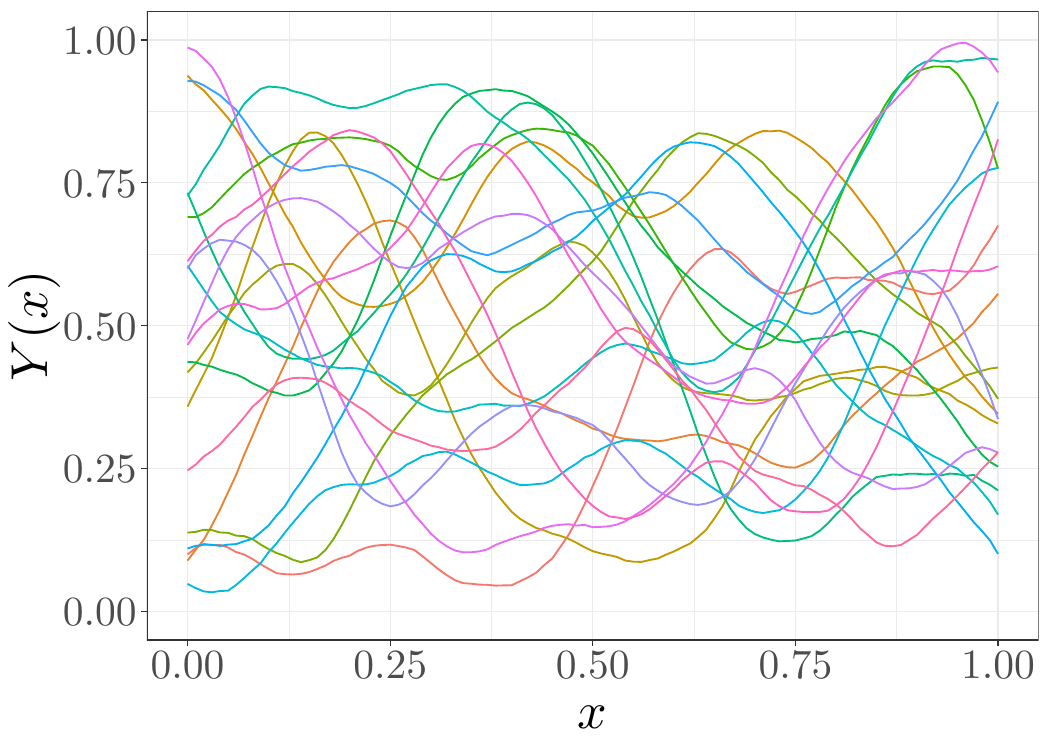}
	\includegraphics[height = 0.285\linewidth]{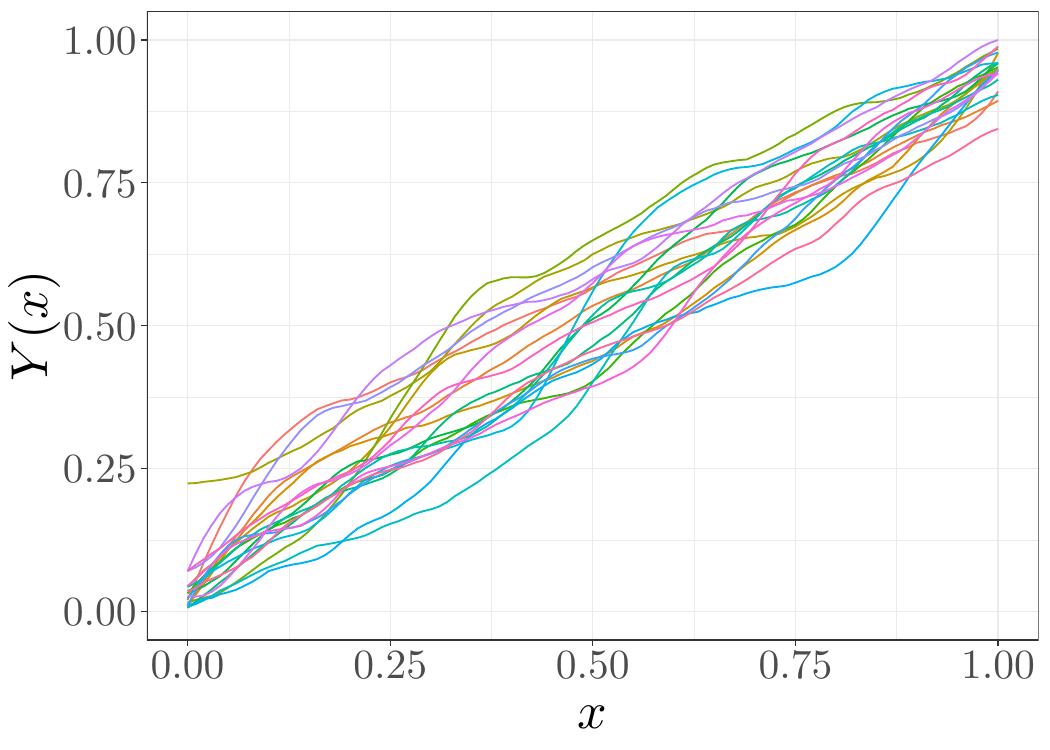}
	\caption{Random GP replicates under boundedness (left) and both boundedness and monotonicity constraints (right) used in the experiments in Section~\ref{sec:numexp:subsec:fixBeta}. As boundedness constraints, we consider $0 \leq Y(x) \leq 1$, for all $x \in [0,1]$.}
	\label{fig:toyExample1_Samples}
\end{figure}
\begin{figure}[t!]
	\centering
	%    \bigskip
	%
	\includegraphics[height = 0.24\linewidth]{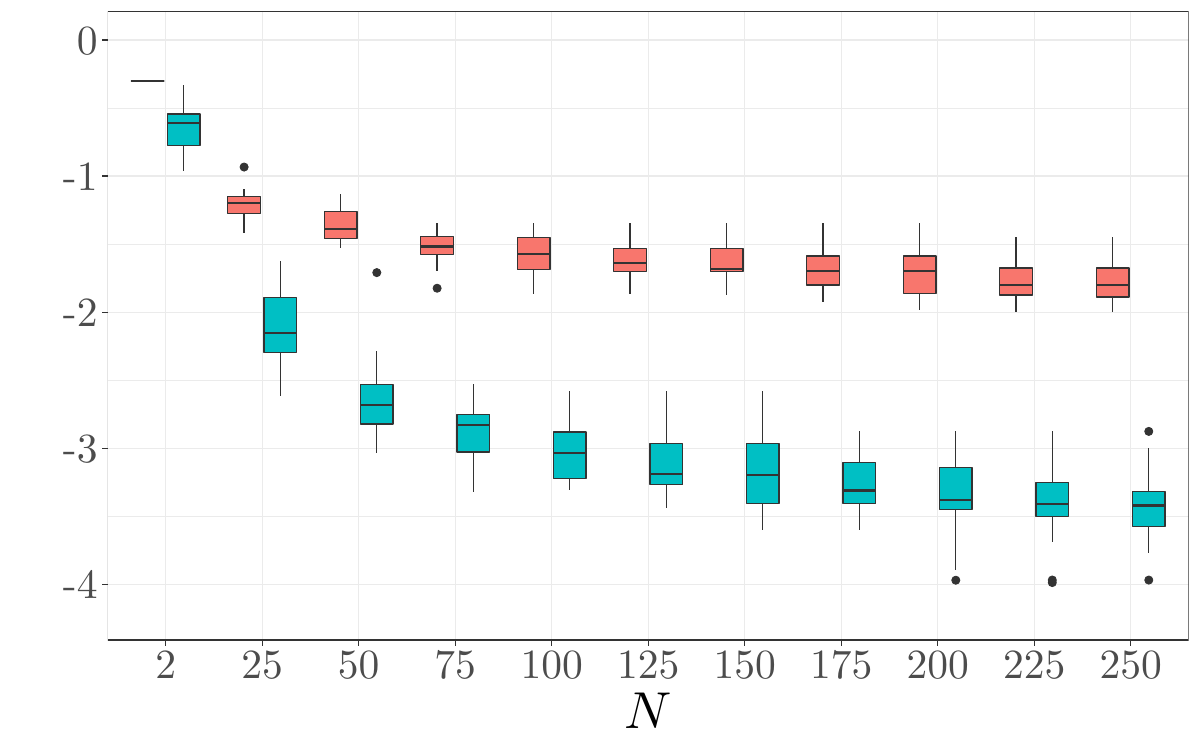}
	\includegraphics[height = 0.24\linewidth]{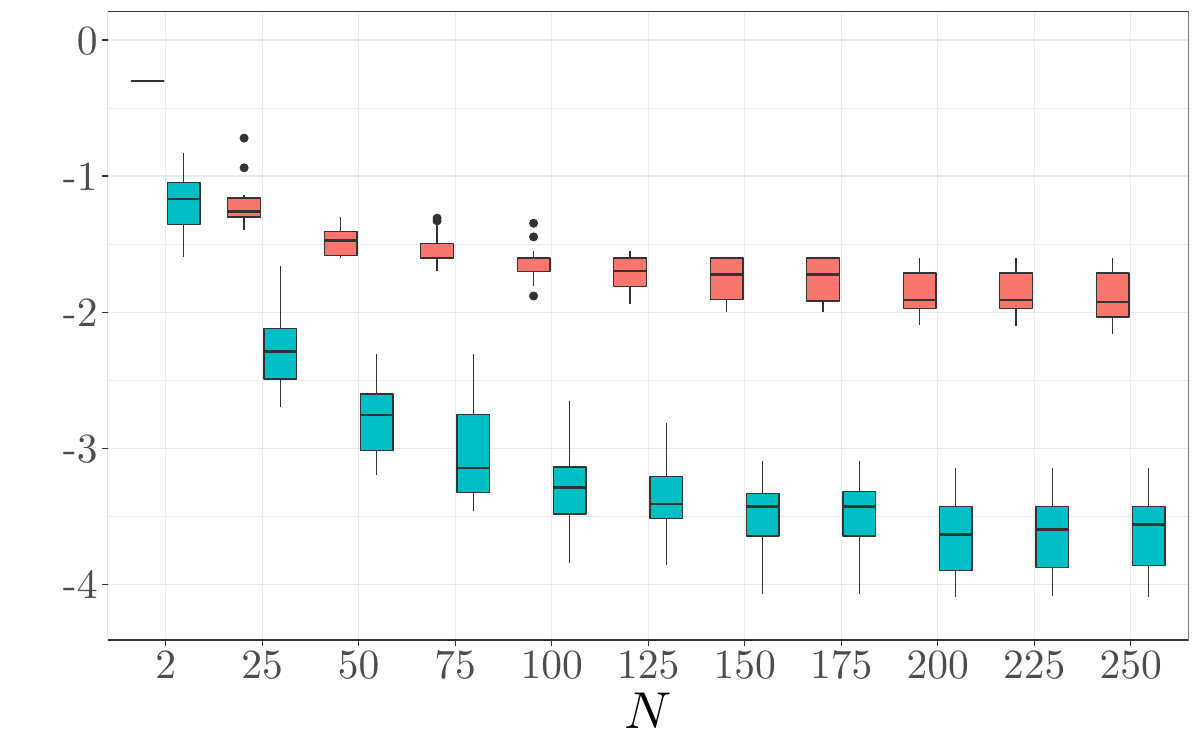}
	\vspace{-3.3ex}
	
	\includegraphics[height = 0.24\linewidth]{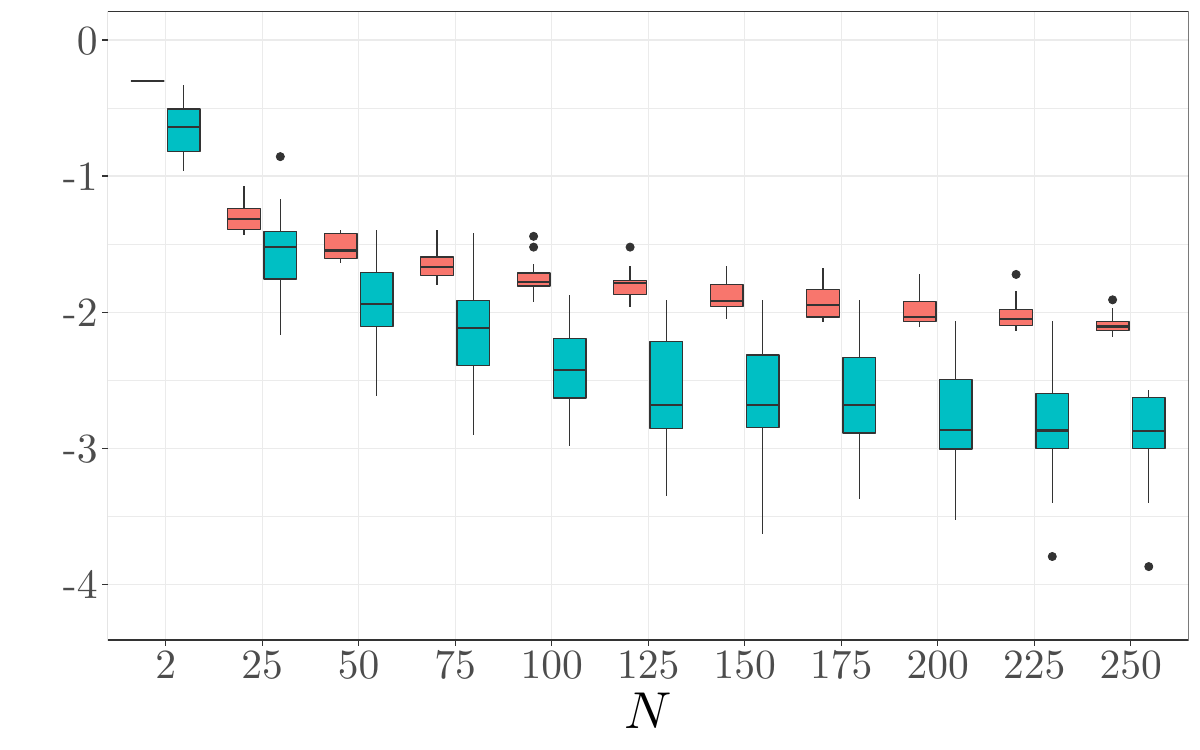}
	\includegraphics[height = 0.24\linewidth]{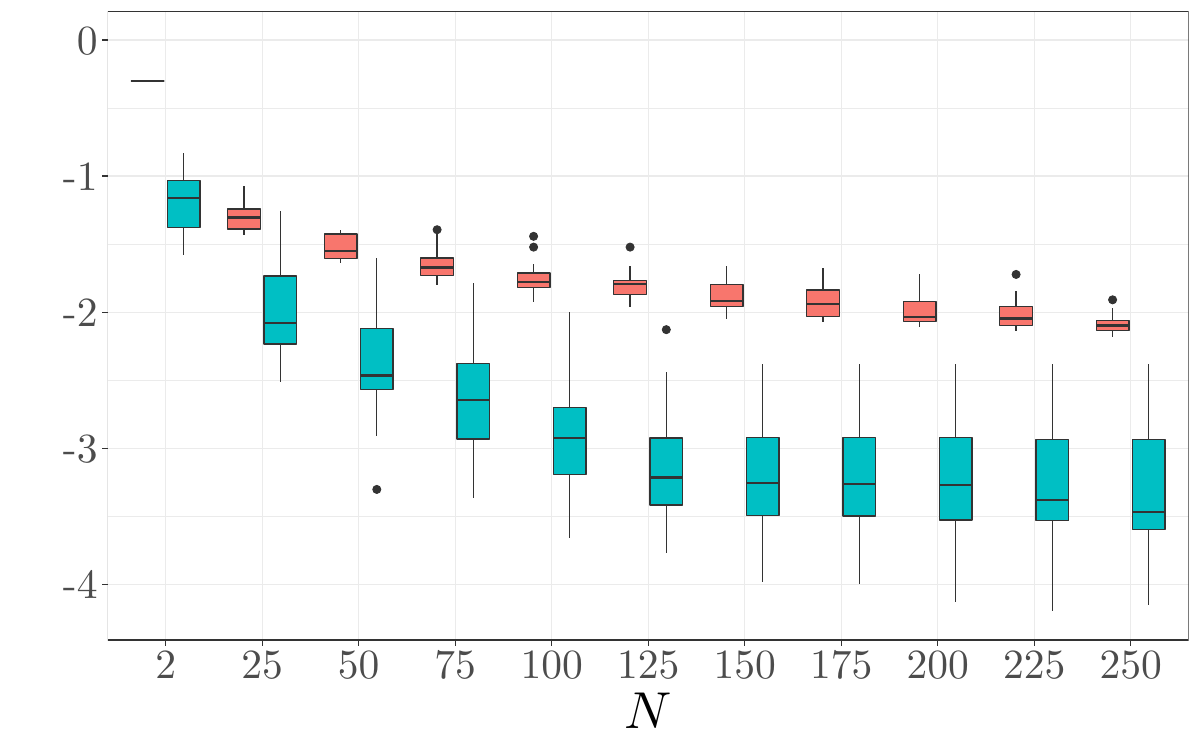}
	\caption{Boxplots of the error $\log_{10} \|\widehat{u}_{N, F} - \widehat{u}_{F}\|_{\infty}$ (blue) and the grid size $\delta_N$ (red) as a function of the number of knots $N$. The boxplots are computed for the twenty replicates in Figure~\ref{fig:toyExample1_Samples}, considering cases when the grid of knots is dense or not (top and bottom, respectively). Results are presented for examples under boundedness (left) or both boundedness and monotonicity (right) constraints.}
	\label{fig:toyExample1_Boxplots}
\end{figure}

\begin{remark}
	To assess the robustness of Theorems~\ref{label:UB:no:extra:assumption} and~\ref{th4} against higher noise levels, we repeated the experiments in Figure~\ref{fig:toyExample1_Boxplots} while increasing the noise variance by a factor of ten, i.e., $\tau = 5\times10^{-1}$. The (unreported) experiments have shown that although the values of $\|\widehat{u}_{N, F} - \widehat{u}_{F}\|_{\infty}$ and $\delta_N$ are slightly higher under increased noise variance, they still exhibit overall decreasing trends as $N$ increases. This confirms the robustness of both theorems.
\end{remark}

\subsection{Error bounds with different regularity assumptions}
\label{sec:numexp:subsec:varyBeta}
We now perform experiments with dense grids while varying $\beta$, the parameter related to the regularity of the kernel. To do so, we consider $\nu = 1/4, 3/8, 1/2, 3/4, 5/2$ (respectively, $\beta = 1/2, 3/4, 1, 1, 1$). Although $\beta = 1$ for $\nu = 1/2, 3/4, 5/2$, we opt to conduct experiments with these values to observe the impact of the GP samples' smoothness order on the rate of error decrease. We follow the same procedure as described in~Section~\ref{sec:numexp:subsec:fixBeta} with the same GP parameters except for the length-parameter that we have increased to $\ell = 0.8$ to control the variability of the samples (see~Figure~\ref{fig:toyExample2_Boxplots}, left panels). This choice seeks to have visible convergence trends for $N \leq N_{\max} = 250$. As monotonicity is unlikely to be satisfied for Mat\'ern kernels with $\nu \leq 1/2$, we focus here on boundedness constraints. To assess if narrower bounds may have an impact in the results, we suggest GP replicates satisfying $0 \leq Y(x) \leq 0.5$. %We apply the MaxMod algorithm for the addition of knots.

\begin{figure}[t!]
	\centering
	% 	\includegraphics[height = 0.235\linewidth]{demToyLineqMaxModGP3_nbExp1samples.pdf}
	% 	\includegraphics[height = 0.235\linewidth, width = 0.4\linewidth]{demToyLineqMaxModGP3_nbExp1ErrorBoxplots.pdf}
	% 	\vspace{-2ex}
	
	% 	\includegraphics[height = 0.235\linewidth]{demToyLineqMaxModGP3_nbExp2samples.pdf}
	% 	\includegraphics[height = 0.235\linewidth, width = 0.4\linewidth]{demToyLineqMaxModGP3_nbExp2ErrorBoxplots.pdf}
	% 	\vspace{-2ex}
	
	% 	\includegraphics[height = 0.235\linewidth]{demToyLineqMaxModGP3_nbExp3samples.pdf}
	% 	\includegraphics[height = 0.235\linewidth, width = 0.4\linewidth]{demToyLineqMaxModGP3_nbExp3ErrorBoxplots.pdf}
	% 	\vspace{-2ex}
	
	% 	\hspace{0.7ex}\includegraphics[height = 0.235\linewidth]{demToyLineqMaxModGP3_nbExp4samples.pdf}
	% 	\hspace{0.7ex}
	% 	\includegraphics[height = 0.235\linewidth, width = 0.39\linewidth]{demToyLineqMaxModGP3_nbExp4ErrorBoxplots.pdf}
	% 	\vspace{-2ex}
	
	% 	\hspace{-1ex}\includegraphics[height = 0.235\linewidth]{demToyLineqMaxModGP3_nbExp5samples.pdf}
	% 	\includegraphics[height = 0.235\linewidth, width = 0.39\linewidth]{demToyLineqMaxModGP3_nbExp5ErrorBoxplots.pdf}
	\includegraphics[width=0.72\linewidth]{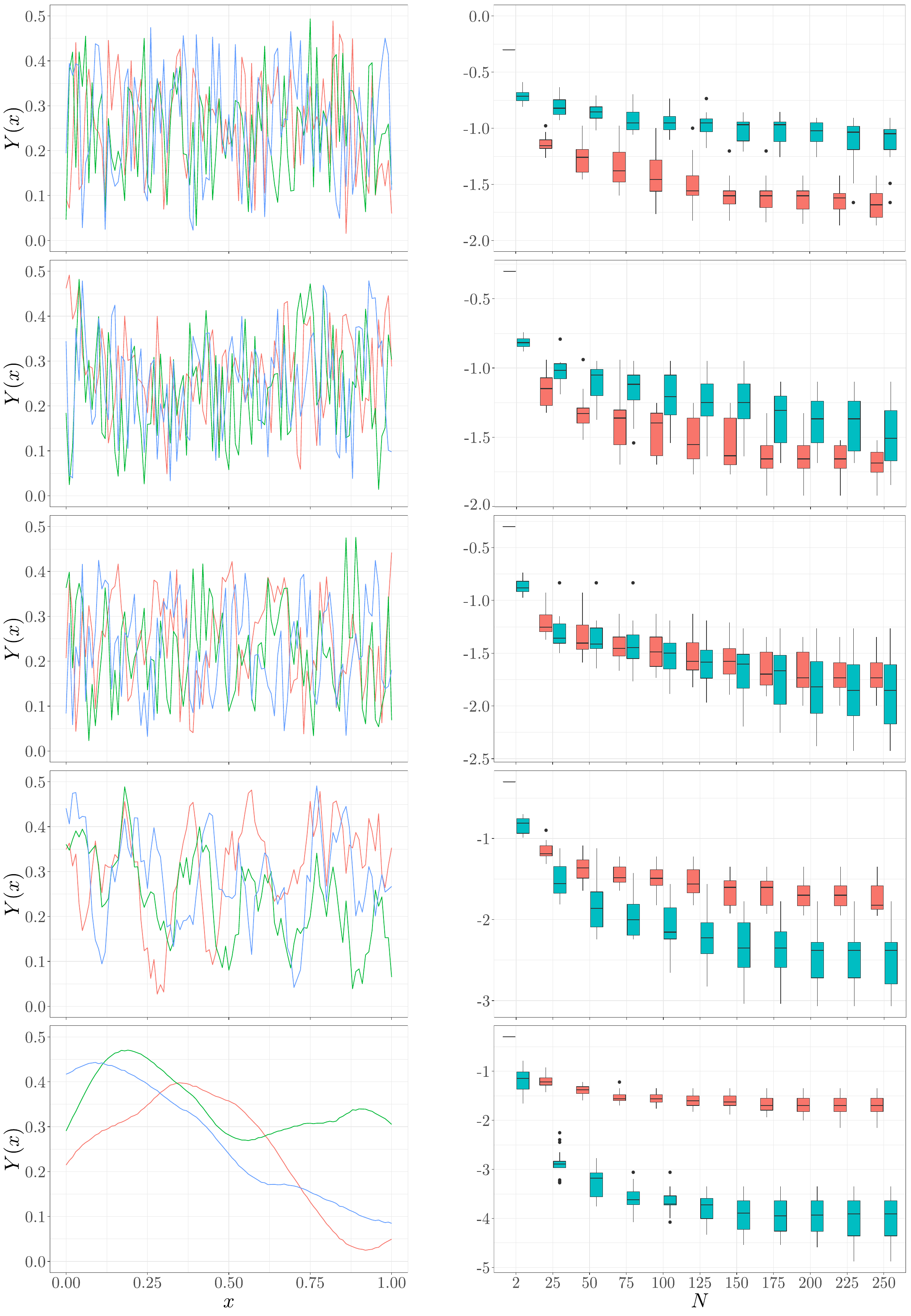}
	
	\caption{Right: boxplots of $\log_{10} \|\widehat{u}_{N, F} - \widehat{u}_{F}\|_{\infty}$ (blue) and $\delta_N$ (red) as a function of $N$. The boxplots are computed for twenty constrained GP replicates using a Mat\'ern kernel with $\nu =  1/4, 3/8, 1/2, 3/4, 5/2$ (from top to bottom). For a better visualization, we only display three of the twenty replicates (left).}
	\label{fig:toyExample2_Boxplots}
\end{figure}
In line with Figure~\ref{fig:toyExample1_Boxplots}, Figure~\ref{fig:toyExample2_Boxplots} shows a decreasing trend of the error $\|\widehat{u}_{N, F} - \widehat{u}_{F}\|_{\infty}$ as $\delta_N$ decreases independently of the value of $\beta$. In particular, we note that the error decreases faster as $\beta$ increases, which is consistent with Theorems~\ref{label:UB:no:extra:assumption} and~\ref{th4} as the asymptotic error bounds there become smaller as $\beta$ increases. In addition, the rate of decrease is higher for smoother GP samples (i.e. for larger $\nu$ values). This pattern has also been verified for the squared exponential kernel, i.e. when $\nu \to \infty$, in experiments unreported here. 

%%%%%%%%%%%%%%%%%%%%%%%%%%%%%%%%%%%%%%%%%%
\section{Conclusions}
\label{conclude} 
%%%%%%%%%%%%%%%%%%%%%%%%%%%%%%%%%%%%%%%%%%%%%%%
We have introduced a general error bound (see Theorem~\ref{th4}) for the constrained optimal smoothing problem and for the equivalent formulation with the MAP estimator. We show that this bound depends on the grid size, the regularity of the kernel, and the distance from the kernel interpolant of the approximation to the set of constraints. In particular, if the kernel interpolant satisfies the constraints, then the latter distance is zero, and the error bound is given by $\mathcal{O}( \delta_N^{\beta/4})$ (see Theorem~\ref{label:UB:no:extra:assumption}). Convergence results are provided for non-equispaced and non-dense grids of knots, allowing certification of sequential schemes, such as the MaxMod algorithm in~\citep{bachoc2022sequential}, introduced for the efficient allocation of knots. To the best of our knowledge, our theoretical results are the first to provide quantitative error bounds for numerical approximations of constrained GPs.

Our results are illustrated numerically through various synthetic examples that account for different types of inequality constraints (e.g., boundedness and monotonicity) and regularity assumptions (i.e. smoothness of the GP samples). Additionally, we examine scenarios with both dense and non-dense grids of knots. Our experiments show that the approximation error decreases as the grid size decreases, with a faster rate of decrease observed for smoother GP samples. This is in line with our theoretical analysis.

\paragraph{Acknowledgments.} We thank Olivier Roustant (IMT, France) for his contribution to this work. Indeed, this study has begun thanks to his interest about the subject. Scientific exchanges with him enabled the improvement of the quality of this work. This work has been supported by the projects GAP (ANR-21-CE40-0007) and GAME (ANR-23-CE46-0007) of the French National Research Agency (ANR). We are grateful to the anonymous reviewers for their constructive feedback, which significantly enhanced the quality of this paper.

\appendix
%%%%%%%%%%%%%%%%%%%%%%%%%%%%%%%%%%%%%%%%%%
\section{Remaining proofs}\label{sectionproofs} 
%%%%%%%%%%%%%%%%%%%%%%%%%%%%%%%%%%%%%%%%%%%%%%%

In order to make the paper self-contained, we provide in this appendix the proofs that are primarily technical or containing pre-existing concepts from other works.

%\noindent {\bf Proof of   Proposition~\ref{prop:psif} }
\begin{proof}[Proof of Proposition~\ref{prop:psif}]
	Note that this result may already be known, but since we haven't found a proper reference with the full proof, we provide it for self-containing. Let $f$ be continuous on $S$. Suppose that $\lim_{\delta \to 0} \Psi_f(\delta)=0$ is not true. Hence, there exists a sequence $(\delta_n)_{n \in \mathbb{N}}$ such that $\delta_n \to 0$ and a constant $A > 0$ such that $\Psi_f(\delta_n) > A$ for all $n \in \mathbb{N}$. Moreover, there exists a sequence $(t_n)_{n \in \mathbb{N}}$ with $t_n >1$ such that 
	\begin{equation}\label{hyp} 
		\dfrac{M_f(t_n\delta_n)}{t_n} > A.  
	\end{equation}
	
	If $(t_n)_{n \in \mathbb{N}}$ has a bounded  subsequence, then $1\leq t_{\phi(n)}\leq b$ and $\frac{M_f(t_{\phi(n)}\delta_n)}{t_{\phi(n)}}\leq M_f(b\delta_n)$. As $f$ is continuous, $M_f(b\delta_n)$ corresponds to a modulus of continuity, so $\lim_{\delta_n \to 0} M_f(b\delta_n)=0$. This leads to a contradiction with~\eqref{hyp}. If $\lim_{n \to \infty} t_n=\infty$, then $\frac{M_f(t_n\delta_n)}{t_n} \leq \frac{M_f(1)}{t_n}$, and so $ \lim_{\delta_n \to 0} \frac{M_f(t_n\delta_n)}{t_n}=0$. Hence, there is a contradiction with~\eqref{hyp}.
	\medskip
	
	We next analyze the case where $f$ is $\beta$-H\"older continuous.
	Then $M_f(\delta)\leq c_f \delta^{\beta}$.  Hence, 
	\begin{align*}
		\frac{M_f(t\delta)}{t} &\leq c_f \delta^{\beta} t^{\beta-1}\leq  c_f \delta^{\beta}, \qquad \quad \text{for $t\geq 1$ and $t\delta \leq 1$,}
		\\
		\dfrac{M_f(t\delta)}{t} &=\dfrac{M_f(1)}{t} \leq \dfrac{c_f}{t} \leq c_f \delta, \qquad \text{for $t\geq 1$ and $t\delta \geq 1$.}
	\end{align*}
	As $\delta \leq 1$ and $ 0< \beta \leq 1$, we have that $\Psi_f(\delta) \leq c_f \delta^{\beta}$.
\end{proof}
%    \item Let $f$ be continuously differentiable, then $| f(s)-f(t)| \leq d_f | s-t |$ so that $M_f(\delta)\leq d_f \delta$.  Hence, for $t\geq 1$ and $t\delta \leq 1$,
$$
%\dfrac{M_f(t\delta)}{t} \leq d_f \delta,
$$
%        and for $t\geq 1$ and $t\delta \geq 1$,
$$
%\dfrac{M_f(t\delta)}{t} =\dfrac{M_f(1)}{t} \leq \dfrac{d_f}{t} \leq d_f \delta.
$$
%    Therefore, $\Psi_f(\delta) \leq d_f \delta$.
%  \end{itemize}

%\medskip

%\noindent  {\bf Proof of \Cref{prop3} }
%\begin{proof}[Proof of \Cref{prop3}]
%   We provide here the main arguments of the proof. From a standard result for finite-dimensional spaces  \citep[][Chapter 6, \S 1]{berlinet2011reproducing}, it holds that $\HN$ is an RKHS with kernel given by:
%%  $$
%     K_N(x, x') = \sum_{i,j=1}^N \phi_i(x) \phi_j(x') \left( G^{-1} \right)_{i,j},
%$$
%where $G$ is the Gram matrix $G = \left( \langle \phi_i, \phi_j \rangle_N \right)_{i,j} = \Gamma_N^{-1}$. Thus, $\left( G^{-1} \right)_{i,j} = K(t_i, t_j)$, which proves the result for $\HN$. The proof for $\HNF$ is identical.
%\end{proof}

%\noindent 	 {\bf Proof of   Proposition~\ref{intermed1} }
\begin{proof}[Proof of Proposition~\ref{intermed1}]
From \eqref{eq:JNF}, we have
\begin{align*}
	J_{N,F}(\pi_N (\widehat{u}_{F}))
	&= \| \pi_N(\widehat{u}_{F})\|_{N}^2+\displaystyle \frac{1}{\tau} \sum_{i=1}^{n}(P(\pi_N(\widehat{u}_{F}))(x_i)-y_i)^2 
	\\
	&=  \| \pi_N(\widehat{u}_{F})\|_{N}^2-\|  \widehat{u}_{F}\|_{\HF}^2+\|  \widehat{u}_{F}\|_{\HF}^2  +\displaystyle \frac{1}{\tau} \sum_{i=1}^{n}(P(\pi_N(\widehat{u}_{F}))(x_i)-y_i)^2. 
	%\\
	%&= \| \pi_N(\widehat{u}_{F})\|_{N}^2-\|  \widehat{u}_{F}\|_{\HF}^2+J_F(\widehat{u}_{F}) +
	%\frac{1}{\tau} \sum_{i=1}^{n} \left(P(\pi_N(\widehat{u}_{F}))(x_i)-y_i\right)^2 - \left(P (\widehat{u}_{F})(x_i)-y_i\right)^2 
	%\\
	%&= \| \pi_N(\widehat{u}_{F})\|_{N}^2-\|  \widehat{u}_{F}\|_{\HF}^2+J_F(\widehat{u}_{F}) + \epsilon_N,
\end{align*}
Using \eqref{eq:JF}, then
\begin{align*}
	J_{N,F}(\pi_N (\widehat{u}_{F}))
	%&= \| \pi_N(\widehat{u}_{F})\|_{N}^2+\displaystyle \frac{1}{\tau} \sum_{i=1}^{n}(P(\pi_N(\widehat{u}_{F}))(x_i)-y_i)^2 
	%\\
	%&=  \| \pi_N(\widehat{u}_{F})\|_{N}^2-\|  \widehat{u}_{F}\|_{\HF}^2+\|  \widehat{u}_{F}\|_{\HF}^2  +\displaystyle \frac{1}{\tau} \sum_{i=1}^{n}(P(\pi_N(\widehat{u}_{F}))(x_i)-y_i)^2 
	%\\
	=& \| \pi_N(\widehat{u}_{F})\|_{N}^2-\|  \widehat{u}_{F}\|_{\HF}^2+J_F(\widehat{u}_{F}) 
	\\
	&+
	\frac{1}{\tau} \sum_{i=1}^{n} \left(P(\pi_N(\widehat{u}_{F}))(x_i)-y_i\right)^2 - \left(P (\widehat{u}_{F})(x_i)-y_i\right)^2 
	\\
	=& \| \pi_N(\widehat{u}_{F})\|_{N}^2-\|  \widehat{u}_{F}\|_{\HF}^2+J_F(\widehat{u}_{F}) + \epsilon_N,
\end{align*}
with 
$$\epsilon_N = \frac{1}{\tau} \sum_{i=1}^{n} \left[ P(\pi_N(\widehat{u}_{F}))(x_i) - P (\widehat{u}_{F})(x_i)\right] \left[ P(\pi_N(\widehat{u}_{F}))(x_i) + P (\widehat{u}_{F})(x_i)-2y_i    \right].$$

In the following, to simplify the notation, we will denote $t_i^-  =  \max \{ t, t \in S_N, t \leq x_i\}$, $t_i^+ =  \min \{ t, t \in S_N, t \geq x_i\}$, $w_- =  w_{{N}_-}$ and $w_+ =  w_{{N}_+}$.
\medskip

From \eqref{interrd1}, we have the bound
\begin{align*} 
	| P\pi_N(\widehat{u}_{F})(x_i)-P\widehat{u}_{F}(x_i)| 
	&= [\pi_N(\widehat{u}_{F})-\widehat{u}_{F}](t_i^-)w_-(x_i) +   [\pi_N(\widehat{u}_{F})-\widehat{u}_{F}](t_i^+) w_+(x_i) \\
	&\leq \| \pi_N(\widehat{u}_{F})-\widehat{u}_{F}\|_{\infty} \leq d_1 \delta_N^{\beta/2}.
\end{align*} 
Using Proposition~\ref{stability},~\eqref{hilb1} and \eqref{hilb2}, 
\begin{align*} 
	| P(\pi_N(\widehat{u}_{F}))(x_i) | 
	&= | \pi_N(\widehat{u}_{F})(t_i^-)w_-(x_i) +   \pi_N(\widehat{u}_{F}) (t_i^+) w_+(x_i)| 
	%	\\
	\leq \| \pi_N(\widehat{u}_{F})\|_{\infty}
	%	\\
	%	\leq  c \| \pi_N(\widehat{u}_{F})\|_{N} 
	%	\\
	\leq  c \|  \widehat{u}_{F}\|_{\HF},
	\\
	| P (\widehat{u}_{F})(x_i) | 
	&= | \widehat{u}_{F}(t_i^-)w_-(x_i) +   \widehat{u}_{F} (t_i^+) w_+(x_i)| 
	%	\\
	\leq \| \widehat{u}_{F}\|_{\infty} 
	%	\\
	\leq  c \|  \widehat{u}_{F}\|_{\HF}.
\end{align*} 
Therefore, $J_{N,F}(\pi_N (\widehat{u}_{F}))= -\ENhat +J_F(\widehat{u}_F)+\epsilon_N$, with 
$$
| \epsilon_N | \leq d_3 \delta_N^{\beta/2}, \quad \mbox{and } \quad d_3=\frac{2n d_1}{\tau}   \Big(c\|\widehat{u}_{F}\|_{\HF} +\max_i | y_i|\Big).
$$ 
From the isometric property of $\rho_N$ in~\eqref{isorhoN}, we have
\begin{align*} 
	J_{N,F}(\widehat{u}_{N,F}) 
	&= \| \widehat{u}_{N,F}\|_N^2 +\displaystyle \frac{1}{\tau} \sum_{i=1}^{n}(P (\widehat{u}_{N,F})(x_i)-y_i)^2 
	\\
	&= \| \rho_N(\widehat{u}_{N,F})\|_{\HF}^2  +\displaystyle \frac{1}{\tau} \sum_{i=1}^{n}(P (\widehat{u}_{N,F})(x_i)-y_i)^2 
	= J_{F}( \rho_N(\widehat{u}_{N,F}))  + \eta_N,  
\end{align*}
with
\begin{align*}
	\eta_N   
	=& \displaystyle \frac{1}{\tau} \sum_{i=1}^{n} (P(\widehat{u}_{N,F})(x_i)-y_i)^2 - (P( \rho_N(\widehat{u}_{N,F}))(x_i)-y_i)^2 
	\\
	=& \displaystyle \frac{1}{\tau} \sum_{i=1}^{n} \left[ P(\widehat{u}_{N,F})(x_i)- P (\rho_N(\widehat{u}_{N,F}))(x_i)\right] \left[ P(\widehat{u}_{N,F})(x_i) +P( \rho_N(\widehat{u}_{N,F}))(x_i)-2y_i    \right].
\end{align*}
Thanks to \eqref{isorhoN}, \eqref{hilb1} and \eqref{hilb2}, we obtain
\begin{align*}
	| P(\widehat{u}_{N,F}(x_i))| 
	&\leq \| \widehat{u}_{N,F}\|_{\infty}  
	%	\\
	\leq  c\| \widehat{u}_{N,F}\|_{N}, 
	\\
	| P(\rho_N\widehat{u}_{N,F})(x_i)| 
	&\leq \|  \rho_N(\widehat{u}_{N,F})\|_{\infty}
	%	\\
	\leq c \|  \rho_N(\widehat{u}_{N,F})\|_{\HF} 
	%	\\
	\leq c \|   \widehat{u}_{N,F}\|_{N},
\end{align*}
and hence,
\begin{eqnarray*}
	P(\widehat{u}_{N,F}-\rho_N(\widehat{u}_{N,F}))(x_i) = (\widehat{u}_{N,F}-\rho_N(\widehat{u}_{N,F} ))(t^-_i)w_-(x_i)+(\widehat{u}_{N,F}-\rho_N(\widehat{u}_{N,F}))(t^+_i)w_+(x_i).
\end{eqnarray*}
For $t\in F$, 
\begin{align*}
	(\widehat{u}_{N,F}-\rho_N(\widehat{u}_{N,F}))(t) 
	&= \langle \widehat{u}_{N,F}, K_N(\cdot,t)\rangle_N-\langle\rho_N(\widehat{u}_{N,F}), K(\cdot,t)\rangle_{\HF}  
	\\
	&=  \langle\rho_N(\widehat{u}_{N,F}), \rho_N (K_N(\cdot,t))\rangle_{\HF}- \langle \rho_N(\widehat{u}_{N,F}), K(\cdot,t)\rangle_{\HF}  
	\\
	&= \langle\rho_N(\widehat{u}_{N,F}),  \rho_N (K_N(\cdot,t))-K(\cdot,t) \rangle_{\HF},  
	\\
	%\Rightarrow
	| \widehat{u}_{N,F}-\rho_N(\widehat{u}_{N,F})(t) |  
	&\leq  \| \rho_N(\widehat{u}_{N,F})\|_{\HF}  \| \rho_N (K_N(\cdot,t))-K(\cdot,t)\|_{\HF}
	\leq  \|  \widehat{u}_{N,F}\|_{N}  \sqrt{G_N}.
\end{align*}
From Proposition~\ref{convnoyaurho}, 
$J_{N,F}(\widehat{u}_{N,F})= J_{F}( \rho_N(\widehat{u}_{N,F}))  + \eta_N $ where 
\begin{eqnarray*} 
	| \eta_N | \leq d_4  \delta_N^{\beta/2}, \qquad d_4=\frac{2n\sqrt{d_2}}{\tau}  \|\widehat{u}_{N,F}\|_{N} \Big(c\|\widehat{u}_{N,F}\|_{N} +\max_i |y_i| \Big).
\end{eqnarray*}
And thanks to the bound of Proposition \eqref{d0}, we have the result. 
\end{proof}

\begin{proof}[Proof of \eqref{mainlim} in Theorem~\ref{thint}]

Let us set $\widehat{h}^N=\rho_N(\widehat{u}_{N,F})$. As $\pi_N(\widehat{u}_F )\in \HNF \cap C_F$, according to \eqref{inter1} and \eqref{inter2}, and as  $\ENhat \geq 0,$    
\begin{align}
	\| \widehat{h}^N \|_{\HF}^2 
	\leq J_F(\widehat{h}^N) 
	\leq J_{N,F}(\widehat{u}_{N,F}) + | \eta_N|
	\leq J_{N,F}(\pi_N\widehat{u}_{F}) + | \eta_N|
	\leq J_{F}(\widehat{u}_{F}) +| \epsilon_N | + | \eta_N|.
	\label{ine}
\end{align}
Hence, the sequence $(\widehat{h}^N)_{N}$ is bounded in $\HF$ so that, by weak compactness in a Hilbert space,  there exists a subsequence $(\widehat{h}^{N_k})_{k }$ and  $h^* \in \HF$ such that    
\begin{equation}\label{weakcompact}
	\widehat{h}^{N_k} \xrightharpoonup[k \to \infty]{} h^*\in \HF, \qquad \text{(weak convergence).}
\end{equation}

As $\HF$ is an RKHS with kernel $K$, for all $t_i \in  S_N, K(\cdot,t_i)\in \HF $ and
$$
\langle \widehat{h}^{N_k},K(\cdot,t_i)\rangle_{\HF}=h^{N_k}(t_i)
\xrightharpoonup[k \to \infty]{} \langle h^{*},K(\cdot,t_i) \rangle_{\HF}=h^{*}({t_i}).
$$
Therefore, for all $N \geq 1$,  $\pi_N(\widehat{h}^{N_k}) \xrightarrow[k\to \infty]{} \pi_N(h^*) $ in the finite-dimensional space $\HN$.

As $S_{N} \subset S_{N+1}$,  as far as $N_k\geq N$,  $ \pi_N(\widehat{h}^{N_k})= \pi_N(\rho_{N_k}(\widehat{u}_{N_k,F}))= \pi_N(\widehat{u}_{N_k,F})$,  
so that 
$$
\pi_{N}(\widehat{u}_{N_k,F}) \xrightarrow[k\to \infty]{} \pi_N(h^*) \quad \mbox{in} \ \HN.
$$
As $\HN$ is an Hibertian subspace of $E_F$, 
$$
\pi_{N}(\widehat{u}_{N_k,F}) \xrightarrow[k\to \infty]{} \pi_N(h^*) \quad \mbox{in} \  E_F.
$$
As $\pi_N(\widehat{u}_{N_k,F}  ) \in  C_F$  and $C_F$ is closed in $E_F$, so that $\forall N$,  $$\pi_N(h^*) \in  C_F.$$
$C_F$ is closed in $E_F$ and  $\pi_N(h^*) \underset{N\to \infty}\rightarrow h^*$ in $E_F$, then $h^* \in C_F$ so that
$$   J_F(\widehat{u}_F  ) \leq J_F(h^*). $$
Then, as $J_F$ is convex and lower semi continuous and $\widehat{h}^{N_k} \xrightharpoonup[k\to \infty]{} \widehat{h}^*\in \HF$ and thanks to \eqref{ine}, as $\displaystyle\lim_{N \to \infty}\delta_N = 0$, so that by Proposition~\ref{intermed1}, $\displaystyle\lim_{N \to \infty}\epsilon_N = 0$ and $\displaystyle\lim_{N \to \infty}\eta_N = 0$, 
\begin{align*}
	J_F(\widehat{u}_F  ) 
	\leq J_F(h^*)
	%	\\
	&\leq \liminf_k J_F(\widehat{h}^{N_k})
	\\
	&\leq \liminf_k J_{N_k,F}(\widehat{u}_{N_k,F} )
	\\
	&\leq \liminf_k J_{N_k,F}(\pi_{N_k}(\widehat{u}_{F} )) 
	\\
	&\leq \limsup_k J_{N_k, F}(\pi_{N_k}(\widehat{u}_{F} ))  
	%	\\
	\leq J_F(\widehat{u}_F  ).  
\end{align*}
Hence  
\begin{equation*}
	h^*=\widehat{u}_{F}
\end{equation*}
and 
\begin{equation*}
	\lim_{N \to \infty} J_{F}(\widehat{h}^{N_k} )
	= \lim_{N \to \infty} J_{N_k,F}(\widehat{u}_{N_k,F})
	= \lim_{N \to \infty} J_{N_k,F}(\pi_{N_k}(\widehat{u}_{F}))
	= J_F(\widehat{u}_F).
\end{equation*}
The real sequences $(J_{F}(\widehat{h}^{N} )), \ (J_{N,F}(\widehat{u}_{N,F})) \ (J_{N,F}(\pi_{N}\widehat{u}_{F}))$ are bounded and have a unique accumulation point $J_F(\widehat{u}_F)$ so that  
\begin{align}\label{intbound}
	J_{F}(\widehat{h}^{N} ) \xrightarrow[N \to \infty]{}  J_F(\widehat{u}_F),
	\quad
	J_{N,F}(\widehat{u}_{N,F}) \xrightarrow[N \to \infty]{} J_F(\widehat{u}_F),
	\quad
	J_{N,F}(\pi_{N}\widehat{u}_{F}) \xrightarrow[N \to \infty]{} J_F(\widehat{u}_F).
\end{align}
The following bounds 
\begin{align*}
	\| \widehat{u}_{N,F}-\widehat{u}_{F}\|_{\infty}
	& \leq    c \ \|\pi_N (\widehat{u}_{F}) - \widehat{u}_{N,F}\|_{N} + \FNhat,
	\\
	\|\pi_N (\widehat{u}_{F}) - \widehat{u}_{N,F}\|_{N}^2 & \leq   J_F(\widehat{u}_F)-
	J_{F}( \widehat{h}^N)  +| \epsilon_N | +| \eta_N|,
\end{align*} 
where $ \FNhat$ is defined in \eqref{F_Nhat}. It leads to \eqref{mainlim}. 
\end{proof}  

\bibliography{Biblio}

\end{document}